\documentclass{amsart}

\usepackage{mathptmx}
\usepackage[T1]{fontenc}
\usepackage{amsmath,amsthm,amsfonts,amssymb}
\usepackage{xcolor}
\usepackage{bbm}

\allowdisplaybreaks

\newtheorem{theorem}{Theorem}
\numberwithin{theorem}{section}
\newtheorem{corollary}[theorem]{Corollary}
\newtheorem{lemma}[theorem]{Lemma}
\newtheorem{proposition}[theorem]{Proposition}
\theoremstyle{remark}
\newtheorem{remark}[theorem]{Remark}

\usepackage{hyperref}
\hypersetup{
	colorlinks=true,
	linktocpage=true,
	linkcolor=[RGB]{161,1,49},
	citecolor=[RGB]{10,136,169},      
	urlcolor=cyan,
}

\newcommand{\R}{\mathbb{R}}
\newcommand{\Z}{\mathbb{Z}}
\newcommand{\N}{\mathbb{N}}

\newcommand{\annil}[1]{{#1}^\perp}
\newcommand{\preannil}[1]{{}^\perp{#1}}
\newcommand{\wklim}{\operatorname{w-}\lim}
\newcommand{\wkslim}{\operatorname{w*-}\lim}

\newcommand{\spanning}{\operatorname{span}}
\newcommand{\wkscspanning}{\overline{\spanning}^{*}}

\newcommand{\acts}{\curvearrowright}

\newcommand{\ran}{\operatorname{ran}}
\newcommand{\Fix}{\operatorname{Fix}}
\newcommand{\Cob}{\operatorname{Cob}}
\newcommand{\wCob}{{\overline{\Cob}}}
\newcommand{\wsCob}{\overline{\Cob}^{*}}

\newcommand{\FSet}{\mathcal{F}(S)}

\newcommand{\conv}{\operatorname{conv}}
\newcommand{\cconv}{\overline{\operatorname{conv}}}
\newcommand{\stcconv}{\overline{\operatorname{conv}}^*}
\newcommand{\ext}{\operatorname{ext}}

\title[Continuous pointwise ergodicity for semigroup actions on locally compact spaces]{Continuous pointwise ergodicity for semigroup actions on locally compact spaces}

\author{Raimundo Briceño}
\address{Facultad de Matem\'aticas, Pontificia Universidad Cat\'olica de Chile. Santiago, Chile}
\email{raimundo.briceno@uc.cl}
\urladdr{http://www.mat.uc.cl/~raimundo.briceno}

\author{Godofredo Iommi}
\address{Facultad de Matem\'aticas, Pontificia Universidad Cat\'olica de Chile. Santiago, Chile}
\email{godofredo.iommi@gmail.com}
\urladdr{http://www.mat.uc.cl/~giommi}

\date{}
\thanks{R.B. was partially supported by ANID/FONDECYT Regular 1240508. G.I. was partially supported by ANID/FONDECYT Regular 1230100. Both authors were partially supported by Avanza UC AV25085.}

\begin{document}

\begin{abstract}
We investigate proper actions of arbitrary semigroups on separable locally compact metric spaces, where point orbits are allowed to escape to infinity. An action is pointwise uniquely ergodic when every compact orbit closure supports exactly one invariant probability measure and non-compact orbit closures support none. The associated ergodic map therefore assigns the selected probability measure to non-escaping points and the zero subprobability to escaping ones.

Under the hypothesis that compact orbit closures admit at least one invariant measure, we establish that the weak* continuity of this ergodic map together with a vanishing at infinity condition is equivalent to the mean ergodicity of the Koopman representation on the space of continuous functions vanishing at infinity. In consequence, every such function and every finite signed measure split uniquely into invariant components and limits of coboundaries. The corresponding projections are obtained by integration against the ergodic map. Because this operator-theoretic characterization avoids explicit averaging schemes, it remains applicable even to semigroups without Følner sequences.

When restricted to countable, discrete, bicancellative, and left amenable semigroups, these properties are shown to be equivalent to the uniform convergence of Følner averages and the weak-star continuity of their dual limits, extending classical results for group actions on compact spaces. Furthermore, we identify the space of ergodic measures with a compactified ergodic quotient, prove that the invariant measure simplex is Bauer, and show that these structural properties descend through proper factor maps. The theoretical framework is complemented by dynamical examples, including a continuously pointwise ergodic subshift that exhibits discontinuous entropy along the ergodic map.
\end{abstract}

\maketitle
\setcounter{tocdepth}{3}
\tableofcontents

\section{Introduction}

Unique ergodicity provides one of the clearest settings in which measure-theoretic and topological dynamics agree. Let $T\colon X\to X$ be a continuous map of a compact metric space. A classical theorem of Oxtoby \cite{oxtoby} states that $T$ has a unique invariant probability measure if and only if, for every $f\in C(X)$,
\[
\frac{1}{n}\sum_{k=0}^{n-1}f(T^kx)
\]
converge uniformly in $x \in X$ to a constant. That constant is then the integral of $f$ against the unique invariant measure, so every point is generic for it.

A natural broader question is what happens when every point is generic for some ergodic invariant measure, which may depend on the point itself. Systems with this property were studied by Dowker and Lederer \cite{dowker-lederer}, and Katznelson and Weiss \cite{katznelson1981}. More recently, Downarowicz and Weiss \cite{downarowicz2020} established several results in the context of compact $\Z$-systems. Related continuity and uniformity questions also arise in the study of uniformity, mean equicontinuity, and continuity properties of F{\o}lner averages. See \cite{hansel1974,cai2022,fuhrmann2022,xu2024,fuhrmann2025} and the references therein.

Motivated by this phenomenon, we seek an intrinsic formulation that does not presuppose an averaging scheme. While equivalent in the abelian setting, the condition that every point is generic for an ergodic measure must generally be distinguished from the requirement that every orbit closure supports exactly one invariant probability measure (see \cite{fuhrmann2025}). We adopt the second property as our foundational concept, which unambiguously assigns a measure to each point based solely on orbit closures rather than specific averaging procedures. An advantage is that this formulation naturally extends to arbitrary semigroups.

Let a semigroup $S$ act by continuous maps on a compact metric space $X$. We say the action is \emph{pointwise uniquely ergodic} if every orbit closure $\operatorname{cl}_X(Sx)$ supports a unique $S$-invariant probability measure, denoted by $\Phi(x)$. This measure is necessarily ergodic, and its uniqueness implies that for every $s \in S$ and $x\in X$ we have $\Phi(sx)=\Phi(x)$. Furthermore, the action is \emph{continuously pointwise ergodic} if the map
\[
\Phi\colon X\longrightarrow M_1^+(X,S)
\]
is weak* continuous, where $M_1^+(X,S)$ denotes the space of $S$-invariant probability measures on $X$. In such a system, while there may be many ergodic invariant measures overall, the specific measure selected by an orbit closure varies continuously with the initial point.

This formulation extends naturally to locally compact spaces. Suppose that $S$ acts by proper maps on a separable locally compact metric space $X$. We say a point $x$ \emph{escapes to infinity} if its orbit closure $\operatorname{cl}_X(Sx)$ is non-compact. In this setting, pointwise unique ergodicity means that every compact orbit closure supports a unique invariant probability measure, while every non-compact orbit closure supports none. The pointwise ergodic map $\Phi$ then takes values in the space of invariant subprobability measures, given by
\[
\Phi(x)=
\begin{cases}
\mu_x, & \text{if } \operatorname{cl}_X(Sx) \text{ is compact},\\
0,     & \text{if } \operatorname{cl}_X(Sx) \text{ is non-compact}.
\end{cases}
\]
As before, the action is \emph{continuously pointwise ergodic} if this map is weak* continuous. For our results in the locally compact setting, we additionally require that $\Phi$ \emph{vanish at infinity}, meaning that the assignment
\[
x\longmapsto\int_X f\,d\Phi(x)
\]
belongs to $C_0(X)$ for every $f\in C_0(X)$, where $C_0(X)$ is the space of continuous functions vanishing at infinity.

Our main objective is to characterize continuous pointwise ergodicity in operator theoretic and measure theoretic terms, both in the compact and the locally compact settings. In previous work, this property has been closely tied to explicit averaging schemes. For compact $\mathbb{Z}$-systems, Downarowicz and Weiss showed that continuous pointwise ergodicity is equivalent to the uniform convergence of ergodic averages \cite{downarowicz2020}. Related extensions and operator convergence formulations for countable discrete amenable groups and more general amenable group actions appear in \cite{xu2024,fuhrmann2025}. 

The present work separates this averaging phenomenon from its underlying operator theoretic structure (see \cite{eisner2015} for an overview on this approach). We demonstrate that the fundamental property is instead the direct sum mean ergodic decomposition of the Koopman representation.  To make this precise, let $\kappa$ be the Koopman representation on $C_0(X)$, let $\Fix(\kappa)$ denote the subspace of its invariant functions, and let $\wCob(\kappa)$ be the norm closure of the linear span of its coboundaries. We say that the action is \emph{mean ergodic} if
\[
C_0(X)=\Fix(\kappa)\oplus\wCob(\kappa).
\]
On the dual space $M(X)=C_0(X)^*$, the adjoint representation $\kappa^*$ is the pushforward representation. Its fixed space is $\Fix(\kappa^*)=M(X,S)$, where $M(X,S)$ denotes the weak* closed space of finite signed $S$-invariant measures, and we write $\wsCob(\kappa^*)$ for the weak* closure of the dual coboundaries. Because this mean ergodic characterization is independent of any averaging scheme, it remains meaningful for semigroups that admit no F{\o}lner sequence. We therefore pass from the previously studied group settings to arbitrary semigroups.

Our first main result concerns compact systems. Assume that every orbit closure supports at least one invariant probability measure. Theorem~\ref{thm:main1} establishes that continuous pointwise ergodicity is equivalent to
\[
C(X)=\Fix(\kappa)\oplus\wCob(\kappa)
\]
and, by duality, to
\[
M(X)=M(X,S)\oplus\wsCob(\kappa^*).
\]
Under these conditions, the mean ergodic projection and its adjoint are given by
\[
(Pf)(x)=\int_X f\,d\Phi(x),
\qquad
P^*\nu=\int_X\Phi(x)\,d\nu(x),
\]
so in particular $P^*\delta_x=\Phi(x)$. The auxiliary measure-valued operator result in Theorem~\ref{thm:psi-general} allows us to construct the preceding projections and identify them as adjoints. We impose no amenability, cancellativity, or countability assumptions on $S$. The existence hypothesis for individual orbits is nevertheless essential. Indeed, in \S\ref{subsec:finite-example}, we provide a finite action for which the mean ergodic decomposition holds and the full system admits a globally unique invariant probability measure, yet some orbit closures support none. 

We obtain the locally compact result via one-point compactification. Let $\widehat X$ denote the one-point compactification of $X$, and let $\infty$ be the added point. Properness ensures that the action extends continuously to $\widehat X$, with $\infty$ fixed by $S$. Under the canonical correspondence between subprobability measures on $X$ and probability measures on $\widehat X$, the zero measure corresponds to the Dirac measure $\delta_\infty$. Thus, the pointwise ergodic map extends to a map
\[
\Phi_\infty\colon\widehat X\longrightarrow M_1^+(\widehat X,S)
\]
for the compactified system. The weak* continuity of $\Phi$, combined with the condition that it vanishes at infinity, is precisely equivalent to the weak* continuity of $\Phi_\infty$. Theorem~\ref{thm:locally-compact-main} therefore generalizes the compact characterization, showing that continuous pointwise ergodicity with vanishing at infinity is equivalent to the corresponding mean ergodic decompositions on $C_0(X)$ and $M(X)$.

Amenability is required only when the mean ergodic projection is realized through concrete averages. Suppose $S$ is a countable, discrete, and bicancellative semigroup that admits a left F{\o}lner sequence. For any non-empty finite set $F\subseteq S$, we define the average
\[
A_F:=\frac{1}{|F|}\sum_{s\in F}\kappa(s).
\]
We say the action is \emph{uniform} if, along every left F{\o}lner sequence $(F_n)_n$, the averages $A_{F_n}f$ converge in the uniform norm for all $f\in C_0(X)$. Furthermore, we call the action \emph{weak* mean continuous} if the corresponding dual limits exist on subprobabilities and depend weak* continuously on the initial measure.

In this setting, a Krylov--Bogolyubov type argument guarantees the existence of invariant probability measures on compact orbit closures. Theorem~\ref{thm:main2} establishes that continuous pointwise ergodicity with vanishing at infinity, mean ergodicity, uniformity, the dual mean ergodic decomposition, and weak* mean continuity are all equivalent properties. Moreover, along any left F{\o}lner sequence, we have
\[
A_{F_n}f\longrightarrow Pf \quad\text{in }\|\cdot\|_\infty, \qquad A_{F_n}^*\nu\longrightarrow P^*\nu \quad\text{weak*}.
\]
Therefore,  the averaging characterizations previously established for group actions extend to this semigroup setting, while the underlying structural equivalence remains valid for arbitrary semigroups.

Continuous pointwise ergodicity also admits a natural quotient interpretation. The fibers of $\Phi_\infty$ induce the compactified ergodic quotient
\[
\pi_\infty\colon\widehat X\longrightarrow \widehat X/\Phi_\infty.
\]
Theorem~\ref{thm:main-quotient} identifies this quotient space with $E_1^+(\widehat X,S)$, the space of ergodic invariant probability measures. Thus,  $M_1^+(\widehat X,S)$ is a Bauer simplex, and the pushforward by $\Phi_\infty$ yields the ergodic decomposition of any invariant probability measure. Letting $\kappa_\infty$ denote the Koopman representation of the compactified action on $C(\widehat X)$, this quotient also characterizes the operator-theoretic decomposition:
\[
\pi_\infty^*C(\widehat X/\Phi_\infty) = \Fix(\kappa_\infty), \qquad \ker((\pi_\infty)_*) = \wsCob(\kappa_\infty^*).
\]

We also investigate proper factor maps. Let $\pi\colon X\to Y$ be a proper factor map, and let $\Phi_X$ and $\Phi_Y$ denote the respective pointwise ergodic maps of the actions on $X$ and $Y$. For countable, discrete, bicancellative, and left amenable semigroups, pointwise unique ergodicity descends from $X$ to $Y$. Continuous pointwise ergodicity, vanishing at infinity, and uniformity descend as well. Moreover, Proposition~\ref{prop:pue-descends-proper-factors} establishes the natural identity
\[
\Phi_Y\circ\pi=\pi_*\circ\Phi_X,
\]
where $\pi_*$ denotes the pushforward of measures. Finally, in Remark~\ref{remark:proper_factors} we provide examples demonstrating that both properness and amenability are essential hypotheses.

The final section illustrates the various phenomena that arise within the theory. We first analyze an interval map that distinguishes pointwise unique ergodicity from continuous pointwise ergodicity, motivating a three-step surgery via invariant deletions. Next, orientation-preserving circle homeomorphisms demonstrate that pointwise unique ergodicity can depend on whether the acting semigroup is $\mathbb{N}$ or $\mathbb{Z}$. Furthermore, rigid rotations provide a setting where the ergodic quotient, the mean ergodic projection, and the invariant and coboundary spaces can be computed explicitly. Translation actions are used to model the complete escape of mass, while a subsequent finite example demonstrates that the existence hypothesis for individual orbits is indispensable. We conclude by constructing a continuously pointwise ergodic subshift whose entropy is discontinuous along the ergodic map, and by relating the dimension of the ergodic quotient to topological emergence.

The paper is organized as follows. Section~\ref{sec:preliminaries} introduces the function and measure spaces, Koopman and pushforward representations, coboundaries, and an auxiliary operator result for measure-valued maps. Section~\ref{sec:compact-case} establishes the main characterizations of continuous pointwise ergodicity for compact systems, while Section~\ref{sec:locally-compact-case} extends them to locally compact spaces using the one-point compactification. Section~\ref{section5} studies F{\o}lner averages for left amenable semigroup actions and their relation to mean ergodicity and weak* mean continuity. Section~\ref{sec:quotients-factors} develops the ergodic quotient, ergodic decompositions, and proper factor maps. Finally, Section~\ref{sec:examples-applications} presents examples and applications illustrating the scope and limitations of the preceding results.

\section{Preliminaries}
\label{sec:preliminaries}

\subsection{Spaces of functions and measures}

Let $X$ be a separable locally compact metric space and denote by $\mathcal{X}$ its corresponding Borel $\sigma$-algebra. Let $C(X)$ be the space of real-valued continuous functions $f\colon X \to \R$, and denote the Banach space of bounded continuous functions on $X$ by
\[
C_b(X):=\{f\in C(X)\mid \|f\|_\infty < \infty\},
\]
where $\|f\|_\infty = \sup_{x \in X}|f(x)|$. A function $f \in C_b(X)$ is said to \textbf{vanish at infinity} if, for every
$\epsilon>0$, the set $\{x\in X\mid |f(x)|\geq \epsilon\}$ is compact in $X$. Consider the closed and separable subspace
\[
C_0(X):=\{f\in C_b(X)\mid f\text{ vanishes at infinity}\}.
\]
If $X$ is compact, then $C_0(X)=C(X)$.

By the Riesz--Markov--Kakutani theorem, the continuous dual space $C_0(X)^*$ of $C_0(X)$ is isometrically isomorphic to $(M(X),\|\cdot\|_{\operatorname{TV}})$. Here, $M(X)$ denotes the space of finite signed regular Borel measures on $X$, and $\|\cdot\|_{\operatorname{TV}}$ is the total variation norm defined by $\|\nu\|_{\operatorname{TV}} = |\nu|(X)$, where $|\nu| = \nu^+ + \nu^-$ and $\nu = \nu^+-\nu^-$ is the Jordan decomposition of $\nu\colon \mathcal{X} \to \R$. To emphasize the measure-theoretic structure of the dual space $M(X)$, we denote the duality pairing by
\[
\langle f,\nu \rangle := \int f \, d\nu \quad \text{for}~f \in C_0(X),~\nu \in M(X).
\]

The \textbf{support} of a measure $\nu \in M(X)$, denoted by $\operatorname{supp}(\nu)$, is defined as the (closed) set of all $x \in X$ such that $|\nu|(U) > 0$ for every open neighborhood $U$ of $x$. Notice that $\operatorname{supp}(\nu) = \operatorname{supp}(\nu^+) \cup \operatorname{supp}(\nu^-)$.

Let $M^+(X)$ denote the \textbf{cone of non-negative measures}, defined by
\[
M^+(X) := \{ \nu \in M(X) \mid \nu(A) \geq 0~\text{for all}~A \in \mathcal{X}\}.
\]
The set of \textbf{probability measures} $M^+_1(X)$ is defined by
\[
M^+_1(X) := \{ \nu \in M^+(X) \mid \nu(X) = 1\},
\]
and the set of \textbf{subprobability measures} $M^+_{\leq 1}(X)$ is defined by
\[
M^+_{\leq 1}(X):=\{\nu\in M^+(X)\mid \nu(X)\leq 1\}.
\]
By the Jordan decomposition theorem, $M(X) = M^+(X) - M^+(X) = \spanning(M^+_1(X))$. 

The space $M(X)$ will usually be endowed with the weak* topology $\sigma(M(X),C_0(X))$. Thus, a net $(\nu_i)_i$ converges weak* to $\nu$ if and only if
\[
\int f\,d\nu_i\to \int f\,d\nu\qquad \text{for all } f\in C_0(X).
\]
In particular, the same characterization holds for sequences. We denote sequential weak* convergence by $\wkslim_n \nu_n=\nu$. This topology is weaker than the topology induced by $\|\cdot\|_{\operatorname{TV}}$, and because $X$ is separable, locally compact, and metrizable, the subspace $M^+_{\leq 1}(X)$ is compact and metrizable.

For each $x \in X$, let $\delta_x\colon C_0(X) \to \R$ be defined by $\delta_x(f) = f(x)$. Via the identification of $C_0(X)^*$ with $M(X)$, $\delta_x$ corresponds to the Dirac delta measure at $x$ and belongs to $M^+_1(X)$. Moreover, the map
\[
\iota_X\colon X \longrightarrow M^+_{\leq 1}(X), \qquad x \longmapsto \iota_X(x) := \delta_x,
\]
is an embedding, which we refer to as the \textbf{Dirac embedding}, and
\[
\stcconv(\iota_X(X)) = 
\begin{cases}
M^+_1(X)            &   \text{if $X$ is compact},    \\
M^+_{\leq 1}(X)     &   \text{otherwise},
\end{cases}
\]
where $\stcconv$ denotes the weak* closure of the convex hull. Observe that, in the non-compact case, the weak* topology induced by $C_0(X)$ does not see mass escaping to infinity and weak* limits of probability measures may be subprobability measures.

\subsection{Actions by proper maps}

For a locally compact metric space $Y$, a continuous map $F\colon X\to Y$ is said to be \textbf{proper} if the preimage of every compact subset of $Y$ is compact in $X$. Note that if $X$ is compact, or if $F$ is a homeomorphism, then $F$ is necessarily proper.

Let $S$ be a semigroup. An \textbf{action by proper maps} of $S$ on $X$, denoted by $S \acts X$, is a mapping $(s, x) \mapsto sx$ from $S \times X$ to $X$ such that for each fixed $s \in S$, the transformation $x \mapsto sx$ is continuous and proper, and the composition rule $(st)x = s(tx)$ holds for all $s, t \in S$ and $x \in X$. If $X$ is compact, or if $S$ is a group acting by homeomorphisms, then every action by continuous maps is an action by proper maps.

Given $x \in X$, denote by $Sx = \{sx \mid s \in S\}$ the \textbf{$S$-orbit} of $x$ and by $\operatorname{cl}_X(Sx)$ its closure inside $X$. A subset $A \subseteq X$ is \textbf{(forward) $S$-invariant} if $sA \subseteq A$ for all $s \in S$. If $A\subseteq X$ is $S$-invariant, then the restricted action on $A$, $S \acts A$, is well-defined. When $A$ is closed, it is again an action by proper maps on a locally compact metric space, and in this case we call it a \textbf{subsystem}. In particular, $S \acts \operatorname{cl}_X(Sx)$ is a subsystem for every $x$.

The \textbf{pushforward} of a measure $\nu \in M(X)$ via $s \in S$ is the measure $s_*\nu$ defined by
\[
s_*\nu(A) := \nu(s^{-1}A) \quad \text{for all}~A \in \mathcal{X}.
\]
A measure $\nu \in M(X)$ is called \textbf{$S$-invariant} if $s_*\nu = \nu$ for all $s \in S$. Denote by $M(X,S)$ the weak* closed subspace of $S$-invariant measures and by $M^+_1(X, S) := M^+_1(X) \cap M(X, S)$ and $M^+_{\leq 1}(X, S) := M^+_{\leq 1}(X) \cap M(X, S)$ the sets of $S$-invariant probability and subprobability measures, respectively. As for general measures, the Jordan decomposition theorem implies that $M(X,S) = M^+(X,S) - M^+(X,S) = \spanning(M^+_1(X,S))$, where $M^+(X,S) = M^+(X) \cap M(X, S)$. Indeed, if $\nu\in M(X,S)$, then $s_*|\nu|=|\nu|$ for every $s\in S$, since
$|s_*\nu|\leq s_*|\nu|$ and both measures have the same total mass. Hence the Jordan parts $\nu^+$ and $\nu^-$ are $S$-invariant.

Fix an $S$-invariant subprobability measure $\mu \in M^+_{\leq 1}(X, S)$. A set $A \in \mathcal{X}$ is \textbf{essentially $S$-invariant} with respect to $\mu$ if $\mu(A \Delta s^{-1}A) = 0$ for all $s \in S$, and the measure $\mu$ is called \textbf{$S$-ergodic} if every essentially $S$-invariant set $A$ is trivial, meaning $\mu(A) \in \{0,1\}$. In particular, the zero measure $0$ is $S$-ergodic. With this convention, if $\mu\in M^+_{\leq 1}(X,S)$ is $S$-ergodic and $\mu\neq 0$, then $\mu(X)=1$, since $X$ itself is essentially $S$-invariant. We denote by $E^+_1(X,S)$ the set of $S$-ergodic invariant probability measures on $X$. Note that if $\mu$ is $S$-invariant, then $\operatorname{supp}(\mu)$ is both $S$-invariant and essentially $S$-invariant with respect to $\mu$. In particular, if $A \subseteq X$ is closed and $S$-invariant, and $\operatorname{supp}(\mu) \subseteq A$, there is a natural identification between $\{\mu \in M^+_{\leq 1}(X, S) \mid \operatorname{supp}(\mu) \subseteq A\}$ and $M^+_{\leq 1}(A, S)$.

If $Y$ is compact Hausdorff and $S$ acts on $Y$ by continuous maps, then, whenever non-empty, $M^+_1(Y,S)$ is a Choquet simplex
\cite[Lemma~4.3]{converse1969}, and its extreme points are precisely the $S$-ergodic probability measures \cite[Proposition~12.4]{phelps2001}.

\subsection{Linear operators and projections}
\label{sec:linear-operator}

Let $T\colon C_0(X) \to C_0(X)$ be a bounded (or, equivalently, norm continuous) linear operator. The \textbf{range} and \textbf{kernel} of $T$ are the subspaces
\[
\ran(T) := \{Tf \mid f \in C_0(X)\}  \qquad \text{and} \qquad \ker(T) := \{f \in C_0(X) \mid Tf = 0\}.
\]
The kernel is norm closed, while the range need not be closed in general. The \textbf{adjoint} of $T$ is the bounded linear operator $T^*\colon M(X) \to M(X)$ characterized by
\[
\int f \,d(T^*\nu) = \int (Tf) \, d\nu \qquad \text{for all}~f \in C_0(X),~\nu \in M(X).
\]
The adjoint $T^*$ is always \textbf{weak*-to-weak* continuous}, that is, $T^*$ is continuous when $M(X)$ is endowed with the weak* topology. Conversely, every weak*-to-weak* continuous linear operator from $M(X)$ to $M(X)$ is the adjoint of some bounded linear operator $T\colon C_0(X) \to C_0(X)$ (see \cite[Theorem~3.1.11]{megginson2012}).

The \textbf{annihilator} of a subspace $U \subseteq C_0(X)$ is the subspace
\[
\annil{U} := \left\{\nu \in M(X) \, \middle| \, \int f \, d\nu = 0~\text{for all}~f \in U\right\}, 
\]
and the \textbf{pre-annihilator} of a subspace $N \subseteq M(X)$ is the subspace
\[
\preannil{N} := \left\{f \in C_0(X) \, \middle| \, \int{f} \, d\nu = 0~\text{for all}~\nu \in N \right\}.
\]
It is well known that $\preannil{(\annil{U})}$ is the norm-closure of $U$, and $\annil{(\preannil{N})}$ is the weak* closure of $N$ \cite[p. 96]{rudin1991}. The kernel and range of the adjoint $T^*$ are the subspaces
\[
\ran(T^*) := \{T^*\nu \mid \nu \in M(X)\}  \qquad \text{and} \qquad \ker(T^*) := \{\nu \in M(X) \mid T^*\nu = 0\}.
\]
The kernel is weak* closed, while the range need not be weak* closed in general. They satisfy (see \cite[Theorem~4.12]{rudin1991})
\[
\ker(T^*) = \annil{\ran(T)} \qquad \text{and} \qquad \ker(T) = \preannil{\ran(T^*)}.
\]

A \textbf{linear projection} is an idempotent linear operator. If $P\colon C_0(X) \to C_0(X)$ is a norm continuous linear projection, then
\[
C_0(X) = \ran(P) \oplus \ker(P),
\]
that is, $C_0(X) = \ran(P) + \ker(P)$ and $\ran(P) \cap \ker(P)=\{0\}$ (see \cite[\S~5.15]{rudin1991}). In this case, $P^*$ is also a linear projection, and
\[
M(X) = \ran(P^*) \oplus \ker(P^*).
\]

A linear projection $Q\colon M(X) \to M(X)$ is weak*-to-weak* continuous if and only if $Q=P^*$ for some norm continuous linear projection $P\colon C_0(X) \to C_0(X)$. This condition is equivalent to requiring that both $\ran(Q)$ and $\ker(Q)$ be weak* closed subspaces of $M(X)$ (see \cite[Exercise 5.18]{fabian2001}).

\subsection{Koopman representation}

The \textbf{opposite semigroup} $S^{\operatorname{op}}$ of $(S, \cdot)$ is defined as $(S,\star)$, where $s \star t := t \cdot s$ for all $s, t \in S$. Given a Banach space $V$, denote by $\mathcal{B}(V)$ the space of bounded linear operators from $V$ to $V$. The \textbf{Koopman representation} associated to an action by proper maps $S \acts X$ is the linear representation
\[
\kappa\colon S^{\operatorname{op}} \longrightarrow \mathcal{B}(C_0(X))
\]
defined by
\[
(\kappa(s)f)(x) = f(sx) \qquad \text{for all}~s \in S,~f \in C_0(X),~x \in X.
\]
Note that $\kappa$ is well-defined because the function $\kappa(s)f$ always belongs to $C_0(X)$. Indeed, $x \mapsto f(sx)$ is continuous since it is a composition of continuous maps. Furthermore, for any $\epsilon > 0$, the set $\{x\in X\mid |f(sx)|\geq \epsilon\}$ is compact, as it is the preimage of the compact set $\{y\in X\mid |f(y)|\geq \epsilon\}$ under the proper map $x \mapsto sx$. Clearly, $\kappa$ is \textbf{positive} (if $f \geq 0$, then $\kappa(s)f \geq 0$), \textbf{multiplicative} ($\kappa(s)(fg) = (\kappa(s)f)(\kappa(s)g)$), and \textbf{non-expansive} ($\|\kappa(s)f\|_\infty \leq \|f\|_\infty$). If $X$ is compact, then $\mathbbm{1}\in C_0(X)=C(X)$, so $\kappa$ is also \textbf{unital} ($\kappa(s)\mathbbm{1} = \mathbbm{1}$). In general, however, $\mathbbm{1}\notin C_0(X)$ when $X$ is non-compact.

The subspace of \textbf{$\kappa$-invariant functions} is defined as
\[
\Fix(\kappa) := \left\{f \in C_0(X) \mid \kappa(s)f = f~\text{for all}~s \in S\right\},
\]
and the subspace of \textbf{$\kappa$-coboundaries} is defined as
\[
\Cob(\kappa) := \operatorname{span}\{f - \kappa(s)f \mid f \in C_0(X),~s \in S\}.
\]
Note that $\Fix(\kappa)$ and $\Cob(\kappa)$ are $\kappa$-invariant subspaces of $C_0(X)$. Furthermore, $\Fix(\kappa)$ is closed in the norm topology. We denote by $\wCob(\kappa)$ the norm closure of the space of $\kappa$-coboundaries, which is also $\kappa$-invariant. We say that $S \acts X$ is \textbf{mean ergodic} if
\[
C_0(X) = \Fix(\kappa) \oplus \wCob(\kappa).
\]
Since both $\Fix(\kappa)$ and $\wCob(\kappa)$ are norm closed subspaces, by \cite[\S~10.3, Theorem~3]{kadets2018} there exists a norm continuous linear projection $P\colon C_0(X) \to C_0(X)$ such that
\[
\ran(P) = \Fix(\kappa) \qquad \text{and} \qquad \ker(P) = \wCob(\kappa).
\] 
The map $P$ will be called the \textbf{mean ergodic projection}.

The \textbf{dual Koopman representation} is the dual linear representation
\[
\kappa^*\colon S \longrightarrow \mathcal{B}(M(X))
\]
defined by
\[
\kappa^*(s)\nu = s_*\nu \qquad \text{for}~s \in S,~\nu \in M(X),
\]
so
\[
(\kappa^*(s)\nu)(A) = \nu(s^{-1}A) \qquad \text{for every}~A \in \mathcal{X}. 
\]
As a dual linear representation, $\kappa^*$ is characterized by
\[
\langle f,\kappa^*(s)\nu\rangle = \langle \kappa(s)f,\nu\rangle \qquad \text{for}~f\in C_0(X),~\nu\in M(X),~s\in S.
\]
Note that $\kappa^*(s)$ is the adjoint of $\kappa(s)$ for every $s \in S$, that is, $\kappa^*(s) = \kappa(s)^*$. The subspace of \textbf{$\kappa^*$-invariant measures} is defined as
\[
\Fix(\kappa^*) := \left\{\nu \in M(X) \mid \kappa^*(s)\nu = \nu~\text{for all}~s \in S\right\},
\]
and the subspace of \textbf{$\kappa^*$-coboundaries} is defined as
\[
\Cob(\kappa^*) := \spanning\{\nu - \kappa^*(s)\nu \mid \nu \in M(X),~s \in S\}.
\]
We denote by $\wsCob(\kappa^*)$ the weak* closure of the space of $\kappa^*$-coboundaries. Note that $\Fix(\kappa^*)$ equals $M(X,S)$, the (weak* closed) subspace of $S$-invariant measures.

\begin{proposition}
\label{prop:ann-cob-koopman}
Let $S$ be a semigroup, $X$ a separable locally compact metric space, and $S \acts X$ an action by proper maps. For the Koopman representation $\kappa\colon S^{\operatorname{op}} \to \mathcal{B}(C_0(X))$ and the dual Koopman representation $\kappa^*\colon S \to \mathcal{B}(M(X))$, we have
\[
\preannil{\Cob(\kappa^*)} = \Fix(\kappa) \qquad \text{and} \qquad \Fix(\kappa^*) = \annil{\Cob(\kappa)},
\]
and, consequently,
\[
\wsCob(\kappa^*) = \annil{\Fix(\kappa)} \qquad \text{and} \qquad \preannil{\Fix(\kappa^*)} = \wCob(\kappa).
\]
In particular, $C_0(X) = \Fix(\kappa) \oplus \wCob(\kappa)$ if and only if $M(X) = \Fix(\kappa^*) \oplus \wsCob(\kappa^*)$. Moreover, the corresponding linear projection from $M(X)$ onto $\Fix(\kappa^*)$ is the adjoint $P^*$ of the mean ergodic projection $P$.
\end{proposition}

\begin{proof}
First, we prove that $\preannil{\Cob(\kappa^*)} = \Fix(\kappa)$. Observe that
\begin{align*}
\preannil{\Cob(\kappa^*)} &   =   \{f \in C_0(X) \mid \langle f,\nu\rangle = 0~\text{for all}~\nu \in \Cob(\kappa^*)\}    \\
                    &   =   \{f \in C_0(X) \mid \langle f,\nu-\kappa^*(s)\nu\rangle = 0 \text{ for all } \nu \in C_0(X)^*,~s \in S\}    \\
                    &   =   \{f \in C_0(X) \mid \langle f-\kappa(s)f,\nu\rangle = 0~\text{for all}~\nu \in C_0(X)^*,~s \in S\}.
\end{align*}
If $f \in \Fix(\kappa)$, then $\langle f-\kappa(s)f,\nu\rangle = \langle 0,\nu\rangle = 0$, so $f \in \preannil{\Cob(\kappa^*)}$. If $f \notin \Fix(\kappa)$, then there exists $s \in S$ such that $f \neq \kappa(s)f$, which implies that $\langle f-\kappa(s)f,\nu\rangle \neq 0$ for some $\nu \in M(X)$, so $f \notin \preannil{\Cob(\kappa^*)}$. Hence, $\Fix(\kappa) = \preannil{\Cob(\kappa^*)}$. Next, we prove that $\Fix(\kappa^*) = \annil{\Cob(\kappa)}$ in a similar fashion. Observe that
\begin{align*}
\annil{\Cob(\kappa)} &   =   \{\nu \in M(X) \mid \langle f,\nu\rangle = 0~\text{for all}~f \in \Cob(\kappa)\}    \\
                    &   =   \{\nu \in M(X) \mid \langle f - \kappa(s)f,\nu\rangle = 0~\text{for all}~f \in C_0(X),~s \in S\}    \\
                    &   =   \{\nu \in M(X) \mid \langle f,\nu-\kappa^*(s)\nu\rangle = 0~\text{for all}~f \in C_0(X),~s \in S\}.
\end{align*}
If $\nu \in \Fix(\kappa^*)$, then $\langle  f,\nu-\kappa^*(s)\nu\rangle = \langle  f,0\rangle = 0$, so $\nu \in \annil{\Cob(\kappa)}$. If $\nu \notin \Fix(\kappa^*)$, then there exists $s \in S$ such that $\nu \neq \kappa^*(s)\nu$, which implies that $\langle f,\nu-\kappa^*(s)\nu\rangle \neq 0$ for some $f \in C_0(X)$, so $\nu \notin \annil{\Cob(\kappa)}$. Hence, $\Fix(\kappa^*) = \annil{\Cob(\kappa)}$. Finally, to obtain the second claim, notice that
\[
\annil{\Fix(\kappa)} = \annil{(\preannil{\Cob(\kappa^*)})} = \wsCob(\kappa^*) \qquad \text{and} \qquad \preannil{\Fix(\kappa^*)} = \preannil{(\annil{\Cob(\kappa)})} = \wCob(\kappa).
\]

Next, assume that $C_0(X) = \Fix(\kappa) \oplus \wCob(\kappa)$. If $P\colon C_0(X) \to \Fix(\kappa)$ is the mean ergodic projection, then $P^*\colon M(X) \to \Fix(\kappa^*)$ is a weak*-to-weak* continuous linear projection such that
\[
M(X) = \ran(P^*) \oplus \ker(P^*) = \annil{\ker(P)} \oplus \annil{\ran(P)},
\]
where we have used that, as $\ker(P) = \preannil{\ran(P^*)}$ and $\ran(P^*)$ is weak* closed, 
\[
\ran(P^*) = \overline{\ran(P^*)}^{*} = \annil{(\preannil{\ran(P^*)})} = \annil{\ker(P)}.
\]
Since $\ran(P) = \Fix(\kappa)$ and $\ker(P) = \wCob(\kappa)$, we have
\[
M(X) = \annil{\wCob(\kappa)} \oplus \annil{\Fix(\kappa)} = \annil{\Cob(\kappa)} \oplus \annil{\Fix(\kappa)} = \Fix(\kappa^*) \oplus \wsCob(\kappa^*).
\]

Conversely, assume that $M(X) = \Fix(\kappa^*) \oplus \wsCob(\kappa^*)$. Since $\Fix(\kappa^*)$ and $\wsCob(\kappa^*)$ are weak* closed subspaces, by \cite[Exercise 5.18]{fabian2001}, the corresponding projection $Q\colon M(X) \to M(X)$ onto $\Fix(\kappa^*)$ is weak*-to-weak* continuous. Therefore, $Q = \tilde{P}^*$ for some norm continuous linear projection $\tilde{P}\colon C_0(X) \to C_0(X)$. Thus,
\[
C_0(X) = \ran(\tilde{P}) \oplus \ker(\tilde{P}) = \preannil{\ker(Q)} \oplus \preannil{\ran(Q)},
\]
where we have used that, as $\ker(Q) = \annil{\ran(\tilde{P})}$ and $\ran(\tilde{P})$ is norm closed,
\[
\ran(\tilde{P}) = \overline{\ran(\tilde{P})} = \preannil{(\annil{\ran(\tilde{P})})} = \preannil{\ker(Q)}.
\]
Since $\ran(Q) = \Fix(\kappa^*)$ and $\ker(Q) = \wsCob(\kappa^*)$,
\[
C_0(X) = \preannil{\wsCob(\kappa^*)} \oplus \preannil{\Fix(\kappa^*)} = \preannil{(\annil{\Fix(\kappa)})} \oplus \preannil{(\annil{\Cob(\kappa)})} = \Fix(\kappa) \oplus \wCob(\kappa).
\]

Finally, observe that
\[
\ran(P) = \Fix(\kappa) = \preannil{(\annil{\Fix(\kappa)})} = \preannil{\wsCob(\kappa^*)} = \preannil{\ker(Q)} = \ran(\tilde{P}).
\]
Moreover,
\[
\ker(P)=\wCob(\kappa)=\preannil{\Fix(\kappa^*)}=\preannil{\ran(Q)}=\ker(\tilde{P}).
\]
Thus $P$ and $\tilde P$ have the same range and the same kernel, and hence $P=\tilde P$.
\end{proof}

\subsection{Pointwise ergodicity}

Fix an action by proper maps $S\acts X$. A point $x \in X$ \textbf{escapes to infinity} if $\operatorname{cl}_X(Sx)$ is not compact. The action $S\acts X$ is said to be \textbf{pointwise invariant-measure admitting} if $M^+_1(\operatorname{cl}_X(Sx),S)$ is not empty for every $x\in X$ that does not escape to infinity. Moreover, it is said to be \textbf{pointwise uniquely ergodic} if, for every $x\in X$,
\[
|M^+_1(\operatorname{cl}_X(Sx),S)|
=
\begin{cases}
0 & \text{if $x$ escapes to infinity},\\
1 & \text{otherwise}.
\end{cases}
\]
If $x$ does not escape to infinity, we denote by $\mu_x$ the unique element of $M^+_1(\operatorname{cl}_X(Sx),S)$. For such $x$, the measure $\mu_x$ necessarily belongs to $E^+_1(X,S)$. Indeed, let $A$ be essentially $S$-invariant with respect to $\mu_x$.
If $0<\mu_x(A)<1$, then the normalized restrictions,
\[
\mu_A(B):=\frac{\mu_x(B\cap A)}{\mu_x(A)}
\qquad \text{and} \qquad
\mu_{A^c}(B):=\frac{\mu_x(B\cap A^c)}{1-\mu_x(A)},
\]
are $S$-invariant probability measures. Indeed, since $A$ is essentially $S$-invariant, for every Borel set $B$ and every $s\in S$,
\[
\mu_x(A\cap s^{-1}B) = \mu_x(s^{-1}(A\cap B)) = \mu_x(A\cap B),
\]
and hence $s_*\mu_A=\mu_A$, and the same argument applies to $\mu_{A^c}$. Since they are absolutely continuous with respect to $\mu_x$, their supports are contained in $\operatorname{supp}(\mu_x)\subseteq \operatorname{cl}_X(Sx)$. Thus $\mu_A,\mu_{A^c}\in M_1^+(\operatorname{cl}_X(Sx),S)$, contradicting the uniqueness of $\mu_x$. Hence $\mu_x(A)\in\{0,1\}$, so $\mu_x$ is $S$-ergodic. In the pointwise uniquely ergodic case, we define the \textbf{pointwise ergodic map} $\Phi\colon X \longrightarrow M^+_{\leq 1}(X,S)$ by
\[
\Phi(x):=
\begin{cases}
0 & \text{if $x$ escapes to infinity},\\
\mu_x & \text{otherwise}.
\end{cases}
\]

The action $S\acts X$ is said to be \textbf{continuously pointwise ergodic} if it is a pointwise uniquely ergodic action and $\Phi$ is weak* continuous. A weak* continuous map $\Psi\colon X\to M(X)$ \textbf{vanishes at infinity} if, for every
$f\in C_0(X)$, the function
\[
f\Psi\colon X \longrightarrow \R, \qquad x \longmapsto f\Psi(x) := \int{f}\, d\Psi(x),
\]
belongs to $C_0(X)$. A continuously pointwise ergodic action $S \acts X$ \textbf{vanishes at infinity} if the associated map $\Phi$ vanishes at infinity. If $X$ is compact, this condition is automatically satisfied.

The action $S \acts X$ is said to be \textbf{uniquely ergodic} if it is pointwise uniquely ergodic and $\Phi$ is constant. In this case, the action $S \acts X$ is continuously pointwise ergodic, and if $X$ is non-compact, then it vanishes at infinity if and only if $\Phi \equiv 0$. In the non-compact case, this terminology is understood in the above pointwise sense. In particular, the constant value of $\Phi$ may be the zero measure.

\subsection{Characterizations of weak* continuous maps}

A function $\Psi\colon X \to M(X)$ is \textbf{weak* measurable} if, for every $f \in C_0(X)$, the function $f\Psi\colon X \to \R$ is measurable (see \cite[Definition 11.48]{aliprantis2006}). Fix $\nu$ in $M(X)$. A weak* measurable function $\Psi$ is \emph{Gelfand integrable} with respect to $\nu$ if for any set $A \in \mathcal{X}$, there exists $\nu_A \in M(X)$ satisfying
\[
\int f \,d\nu_A = \int_{A} f\Psi \, d\nu \qquad \text{for every } f\in C_0(X).
\]
If $\Psi$ is Gelfand integrable, we denote $\nu_A$ by $\int_{A} \Psi \, d\nu$ for each $A \in \mathcal{X}$.

\begin{lemma}
\label{lem:gelfand}
Let $X$ be a separable locally compact metric space. If $\Psi\colon X \to M(X)$ is weak* continuous and vanishes at infinity, then $\Psi$ is weak* measurable, has norm bounded range, and is Gelfand integrable with respect to every finite positive measure.
\end{lemma}

\begin{proof}
Assume that $\Psi$ is weak* continuous and vanishes at infinity. Then, for any $f \in C_0(X)$, the function $f\Psi$ belongs to $C_0(X)$, hence is measurable, so $\Psi$ is weak* measurable.

For each $x \in X$, via the Riesz--Markov--Kakutani theorem, the measure $\Psi(x)$ can be regarded as the bounded linear functional 
\[
\langle \cdot,\Psi(x)\rangle \colon C_0(X) \longrightarrow \R, \qquad f \longmapsto \langle f,\Psi(x)\rangle = \int f \,d\Psi(x).
\]
Consider the family of bounded linear functionals $\{\langle \cdot,\Psi(x)\rangle\}_{x \in X}$. For each $f \in C_0(X)$, as $f\Psi \in C_0(X)$, we have that
\[
\sup_{x \in X} \left|\langle f,\Psi(x)\rangle\right| = \sup_{x \in X} \left|f\Psi(x)\right| = \|f\Psi\|_\infty < \infty.
\]
Thus, by the Banach–Steinhaus theorem and the isometric identification between $M(X)$ and $C_0(X)^*$, we obtain that
\[
\sup_{x \in X} \|\Psi(x)\|_{\operatorname{TV}} = \sup_{x \in X} \sup_{\|f\|_\infty \leq 1}|\langle f,\Psi(x)\rangle| < \infty,
\]
that is, $\Psi$ has norm bounded range. Finally, \cite[Corollary 11.53]{aliprantis2006} shows that $\Psi$ is Gelfand integrable with respect to any $\nu$ in $M^+(X)$.
\end{proof}

The following result characterizes when a map $\Psi\colon X \to M(X)$ is weak* continuous and vanishes at infinity (see \cite[\S~VI.7.1, Theorem~1]{dunford1988} for related results).

\begin{theorem}
\label{thm:psi-general}
Let $X$ be a separable locally compact metric space. For any map $\Psi\colon X \to M(X)$, the following are equivalent:
\begin{enumerate}
    \item The map $\Psi$ is weak* continuous and vanishes at infinity.
    \item The map
    \[
    P\colon C_0(X) \longrightarrow C_0(X), \qquad f \longmapsto Pf := f\Psi,
    \]
    is a well-defined bounded linear operator.
    \item The map $\Psi$ is Gelfand integrable with respect to every $\nu\in M(X)$ and the map
    \[
    Q\colon M(X) \longrightarrow M(X), \quad \nu \longmapsto Q\nu := \int{\Psi}\, d\nu,
    \]
    is a well-defined weak*-to-weak* continuous linear operator.
\end{enumerate}
If any of the above holds, then $Q = P^*$ and $Q$ is the unique weak*-to-weak* continuous linear operator such that $Q \circ \iota_X = \Psi$, that is, $Q\delta_x = \Psi(x)$ for all $x \in X$. In particular, $\wkscspanning(\Psi(X)) = \overline{\ran(P^*)}^{\,*} = \annil{\ker(P)}$. If, moreover, $\ran(P^*)$ is weak* closed, for instance if $P^2=P$, then $\wkscspanning(\Psi(X))=\ran(P^*)$.
\end{theorem}

\begin{proof}
We prove the equivalences $(1) \iff (2)$ and $(1) \iff (3)$.

\medskip
\noindent
{$(1) \iff (2)$.} Note that
\[
(Pf)(x) = \int f \,d\Psi(x) \qquad \text{for}~f \in C_0(X),~x \in X,
\]
so, by definition of the weak* topology, the map $\Psi$ is weak* continuous and vanishes at infinity if and only if $Pf$ belongs to $C_0(X)$ for each $f \in C_0(X)$ or, equivalently, if the map $P\colon C_0(X) \to C_0(X)$ is well-defined. Therefore, it suffices to check that whenever $P$ is well-defined, it is linear and bounded.

Assume that $P$ is well-defined or, equivalently, that $\Psi$ is weak* continuous and vanishes at infinity. That $P$ is linear follows directly from the linearity of the integral. To prove that $P$ is bounded, set $C :=\sup_{x \in X} \|\Psi(x)\|_{\operatorname{TV}}$ and observe that $C < \infty$ by Lemma~\ref{lem:gelfand}. Then,
\begin{align*}
\|Pf\|_\infty = \sup_{x \in X} \left|\int f \,d\Psi(x)\right| & \leq \sup_{x \in X} \int \left|f\right|\,d|\Psi(x)| \\
            &  \leq \sup_{x \in X} \int \|f\|_\infty\,d|\Psi(x)| = \|f\|_\infty \sup_{x \in X} |\Psi(x)|(X),
\end{align*}
so
\[
\|Pf\|_\infty \leq \sup_{x \in X} \|\Psi(x)\|_{\operatorname{TV}}\|f\|_\infty = C\|f\|_\infty,
\]
which proves that $P$ is a bounded linear operator on $C_0(X)$.

\medskip
\noindent
{$(1) \implies (3)$.} If $\Psi$ is weak* continuous and vanishes at infinity, then Lemma~\ref{lem:gelfand} implies that $\Psi$ is weak* measurable, has norm bounded range, and is Gelfand integrable with respect to every finite positive measure, and therefore, by the Jordan decomposition, with respect to every $\nu\in M(X)$. Thus, $Q\nu = \int{\Psi}\,d\nu$ is well-defined for all $\nu \in M(X)$. That $Q$ is linear follows from the definition of Gelfand integral. Moreover, it was already established that the weak* continuity of $\Psi$ is equivalent to $P$ being a well-defined bounded linear operator. Since, for every $f\in C_0(X)$ and $\nu\in M(X)$,
\[
\langle f,Q\nu\rangle = \left\langle f,\int_X \Psi\,d\nu\right\rangle = \int_X f\Psi \,d\nu = \int Pf\,d\nu
= \langle Pf,\nu\rangle = \langle f,P^*\nu\rangle,
\]
we have that $Q=P^*$, so $Q$ is weak*-to-weak* continuous.

\medskip
\noindent
{$(3) \implies (1)$.} Assume that $\Psi$ is Gelfand integrable with respect to every $\nu \in M(X)$. Observe that $Q \delta_x = \int \Psi \, d\delta_x = \Psi(x)$ for all $x \in X$, so $\Psi = Q \circ \iota_X$. As $Q$ is weak*-to-weak* continuous and $\iota_X$ is weak* continuous, $\Psi$ is weak* continuous as well. Moreover, the weak*-to-weak* continuity of $Q$ implies that there exists a bounded linear operator $\widetilde{P}\colon C_0(X) \to C_0(X)$ such that $Q = \widetilde{P}^*$. Therefore, for any $f \in C_0(X)$ and $x \in X$,
\[
f\Psi(x) = \langle f,\Psi(x)\rangle = \langle f,Q\delta_x\rangle = \langle \widetilde{P}f,\delta_x\rangle = (\widetilde{P}f)(x),
\]
that is $f\Psi = \widetilde{P}f$. This shows that $f\Psi$ belongs to $C_0(X)$, and as $f$ was arbitrary, we have that $\Psi$ vanishes at infinity.

\medskip

Suppose $Q$ is a weak*-to-weak* continuous linear operator satisfying $Q\delta_x = \Psi(x)$ for all $x \in X$. Then, for any $f \in C_0(X)$,
\[
\langle f,P^*\delta_x\rangle = \langle Pf,\delta_x\rangle = (Pf)(x) = \int{f}\,d\Psi(x) = \langle f,\Psi(x)\rangle =  \langle f,Q\delta_x\rangle.
\]
As $f$ was arbitrary, $P^*\delta_x = Q\delta_x$ for all $x \in X$. Since $P^*$ is also weak*-to-weak* continuous, and the linear span of Dirac delta measures, $\spanning(\iota_X(X))$, is weak* dense in $M(X)$, we have that $Q = P^*$.

Since $\spanning(\iota_X(X))$ is weak* dense in $M(X)$ and
$P^*$ is weak*-to-weak* continuous, we have
\[
\ran(P^*) \subseteq \overline{\spanning(P^*\iota_X(X))}^{*} = \wkscspanning(\Psi(X)).
\]
The reverse inclusion follows from $\Psi(X)\subseteq\ran(P^*)$. Therefore
\[
\wkscspanning(\Psi(X)) = \overline{\ran(P^*)}^{*}.
\]
Finally,
\[
\overline{\ran(P^*)}^{*} = \annil{\ker(P)}
\]
by the annihilator identities recalled in Section~\ref{sec:linear-operator}. If $P^2=P$, then $P^*$ is a weak*-to-weak* continuous projection, so $\ran(P^*)=\ker(I-P^*)$ is weak* closed.
\end{proof}

\section{The compact case}
\label{sec:compact-case}

Throughout this section, $X$ is compact. Hence $C_0(X)=C(X)$, no point escapes to infinity, and pointwise unique ergodicity simply means that each orbit closure $\operatorname{cl}_X(Sx)$ supports a unique $S$-invariant probability measure. In this case, the associated map $\Phi$ takes values in $M^+_1(X,S)$, and the vanishing-at-infinity condition is automatic.

\subsection{Almost every point generates the support}

The following result establishes that if $\mu$ is $S$-ergodic, then $\mu$-almost every point generates the support of $\mu$. As a consequence, if the system is pointwise uniquely ergodic, the map $\Phi$ is surjective onto the set of $S$-ergodic measures.

\begin{proposition}
\label{prop:transitive-point}
Let $S$ be a semigroup, $X$ a compact metric space, and $S \acts X$ an action by continuous maps. Then, for any $\mu \in E^+_1(X,S)$ and for $\mu$-almost every $x \in X$,
\[
\operatorname{supp}(\mu) = \operatorname{cl}_X(Sx) \qquad \text{and} \qquad \mu \in M^+_1(\operatorname{cl}_X(Sx),S).
\]
In particular, if $S \acts X$ is pointwise uniquely ergodic, then $\Phi(X) = E^+_1(X,S)$.
\end{proposition}

\begin{proof}
Fix $\mu\in E^+_1(X,S)$, and let $(U_n)_n$ be a countable family of open subsets of $X$ such that $(U_n\cap\operatorname{supp}(\mu))_n$ is a basis for the relative topology of $\operatorname{supp}(\mu)$, with
$U_n\cap\operatorname{supp}(\mu)\neq\emptyset$. For each $n$, define
\[
B_n := \operatorname{supp}(\mu)\cap \bigcap_{t\in S} t^{-1}(X\setminus U_n).
\]
Then $B_n$ is closed, hence Borel, and consists of the points in $\operatorname{supp}(\mu)$ whose $S$-orbit does not meet $U_n$. We claim that $B_n$ is essentially $S$-invariant. Indeed, if $x\in B_n$ and $s\in S$, then for every $t\in S$ we have
\[
t(sx)=(ts)x\notin U_n,
\]
so $sx\in B_n$. Thus $B_n\subseteq s^{-1}B_n$. Since $\mu$ is $S$-invariant, $\mu(s^{-1}B_n)=\mu(B_n)$, and hence $\mu(s^{-1}B_n\setminus B_n)=0$. Therefore $B_n$ is essentially $S$-invariant.

By ergodicity, $\mu(B_n)\in\{0,1\}$. On the other hand, for any fixed $t\in S$, $t^{-1}U_n\subseteq X\setminus B_n$ and $\mu(t^{-1}U_n)=\mu(U_n)>0$, because $U_n$ meets the support of $\mu$. Hence $\mu(B_n)\neq1$, so $\mu(B_n)=0$.

Therefore, for $\mu$-almost every $x\in\operatorname{supp}(\mu)$ and every $n$, the orbit $Sx$ meets $U_n$. Hence $Sx$ is dense in $\operatorname{supp}(\mu)$, and since $\operatorname{supp}(\mu)$ is closed and $S$-invariant, we obtain
\[
\operatorname{cl}_X(Sx)=\operatorname{supp}(\mu)
\]
for $\mu$-almost every $x$. Set
\[
A_\mu := \operatorname{supp}(\mu)\setminus\bigcup_{n\in\mathbb N}B_n.
\]
Then $A_\mu$ is Borel, $\mu(A_\mu)=1$, and every $x\in A_\mu$ satisfies $\operatorname{cl}_X(Sx)=\operatorname{supp}(\mu)$.

Finally, assume that $S \acts X$ is pointwise uniquely ergodic. As $\Phi(x)$ is $S$-ergodic for any $x \in X$, $\Phi(X) \subseteq E^+_1(X,S)$. To prove the reverse inclusion, let $\mu \in E^+_1(X,S)$. Note that the set $A_\mu$ has full $\mu$-measure and is therefore non-empty. Let $x \in A_\mu$. Since $\operatorname{supp}(\mu) = \operatorname{cl}_X(Sx)$, $\mu$ is an $S$-invariant probability measure supported on $\operatorname{cl}_X(Sx)$. By the definition of pointwise unique ergodicity, the system restricted to the orbit closure $\operatorname{cl}_X(Sx)$ admits a unique $S$-invariant probability measure, which is $\Phi(x)$. Thus, we must have that $\mu = \Phi(x)$. This shows that $E^+_1(X,S) \subseteq \Phi(X)$, proving the equality.
\end{proof}

\subsection{Linear and convex hulls of ergodic measures}

Let $C$ be a compact and convex non-empty subset of a Hausdorff locally convex topological vector space $V$. The Krein-Milman Theorem asserts that $C$ is equal to the closed convex hull of the set of extreme points $\ext C$:
\[
C=\cconv(\ext C).
\]
Moreover, if $K \subseteq C$, then $C = \cconv(K)$ if and only if $\ext C \subseteq \operatorname{cl}_V(K)$ (see \cite[Theorem~7.4 and Theorem~7.8]{conway1990}). 

\begin{lemma}
\label{lem:convex}
Let $S$ be a semigroup, let $X$ be a compact metric space, and let
$S\acts X$ be an action by continuous maps. If
$M_1^+(X,S)\neq\emptyset$, then
\[
\wkscspanning\left(E_1^+(X,S)\right) = M(X,S).
\]
In particular, if $S\acts X$ is pointwise uniquely ergodic, then $\wkscspanning\left(\Phi(X)\right) = M(X,S)$.
\end{lemma}

\begin{proof}
By the Krein--Milman theorem, $M_1^+(X,S) = \stcconv\left(E_1^+(X,S)\right)$. Hence $M_1^+(X,S) \subseteq \wkscspanning\left(E_1^+(X,S)\right)$. As $M(X,S)=\spanning\left(M_1^+(X,S)\right)$, we obtain $M(X,S)
\subseteq \wkscspanning\left(E_1^+(X,S)\right)$. The reverse inclusion follows because $M(X,S)$ is weak* closed and
contains $E_1^+(X,S)$.

If the action is pointwise uniquely ergodic, then
Proposition~\ref{prop:transitive-point} gives
$\Phi(X)=E_1^+(X,S)$.
\end{proof}

\subsection{Mean ergodicity and continuous pointwise ergodicity}

\begin{theorem}
\label{thm:main1}
Let $S$ be a semigroup, $X$ a compact metric space, and $S \acts X$ a pointwise invariant-measure admitting action by continuous maps. The following are equivalent:
\begin{enumerate}
\item $S \acts X$ is continuously pointwise ergodic.
\item $C(X) = \Fix(\kappa) \oplus \wCob(\kappa)$.
\item $M(X) = \Fix(\kappa^*) \oplus \wsCob(\kappa^*)$.
\end{enumerate}
If any of these holds, then:
\begin{itemize}
\item The mean ergodic projection $P\colon C(X) \to \Fix(\kappa)$ is defined by
\[
(Pf)(x) = \int f\,d\Phi(x)   \qquad   \text{for}~f \in C(X),~x \in X.
\]
\item The dual mean ergodic projection $Q\colon M(X) \to \Fix(\kappa^*)$ is defined by
\[
Q\nu = \int{\Phi} \, d\nu \qquad   \text{for}~\nu \in M(X),
\]
where the integral is in the sense of Gelfand.
\item The adjoint $P^*$ of $P$ is such that $P^* = Q$ and $P^*\delta_x = \Phi(x)$ for all $x \in X$.
\end{itemize}
\end{theorem}

\begin{proof}
The equivalence between (2) and (3) follows from Proposition~\ref{prop:ann-cob-koopman}. Thus it suffices to prove the equivalence between $(1)$ and $(2)$.

\medskip
\noindent
\textbf{$(1) \implies (2)$.} If $S \acts X$ is continuously pointwise ergodic, the map $\Phi\colon X \to E^+_1(X,S)$ is well-defined and weak* continuous. Then, by Theorem~\ref{thm:psi-general}, $P\colon C(X) \to C(X)$ defined by $(Pf)(x) = \int{f}\,d\Phi(x)$ is a well-defined bounded linear operator.

For every $s\in S$ and $x\in X$, we have
$\operatorname{cl}_X(S(sx))\subseteq\operatorname{cl}_X(Sx)$. Hence
$\Phi(sx)$, viewed as a measure on $\operatorname{cl}_X(Sx)$, is an
$S$-invariant probability measure. By uniqueness on
$\operatorname{cl}_X(Sx)$, $\Phi(sx)=\Phi(x)$. Thus, for any $f \in C(X)$,
\[
(\kappa(s)Pf)(x) = (Pf)(sx) = \int{f}\,d\Phi(sx) = \int{f}\,d\Phi(x) = (Pf)(x),
\]
so $\kappa(s)P = P$. Since $s$ was arbitrary, this implies that $\ran(P) \subseteq \Fix(\kappa)$.

If $h\in\Fix(\kappa)$, then $h(sx)=h(x)$ for every $s\in S$.
By continuity, $h(y)=h(x)$ for every
$y\in\operatorname{cl}_X(Sx)$. Since
$\operatorname{supp}(\Phi(x))\subseteq\operatorname{cl}_X(Sx)$, we get
\[
(Ph)(x)=\int h\,d\Phi(x)=h(x).
\]
Thus $Ph=h$. This shows that $\ran(P) \supseteq \Fix(\kappa)$, and therefore, $\ran(P) = \Fix(\kappa)$. Moreover, for any $f \in C(X)$, $P^2f = P(Pf) = Pf$, so $P^2=P$ and $P$ is a linear projection.

Finally, observe that $h \in \ker(P)$ if and only if 
\[
(Ph)(x) = \int h\,d\Phi(x) = 0 \quad \text{for all}~x \in X,
\]
and, since $\wkscspanning(\Phi(X)) = M(X,S)$ by Lemma~\ref{lem:convex}, this is also equivalent to
\[
\int h \,d\mu = 0 \quad \text{for all}~\mu \in M(X,S).
\]
Therefore, $\ker(P) = \preannil{M(X,S)} = \preannil{\Fix(\kappa^*)} = \wCob(\kappa)$, where the last equality follows from Proposition~\ref{prop:ann-cob-koopman}. We conclude that $C(X) = \ran(P) \oplus \ker(P) = \Fix(\kappa) \oplus \wCob(\kappa)$.

\medskip
\noindent
\textbf{$(2) \implies (1)$.} Assume that $C(X) = \Fix(\kappa) \oplus \wCob(\kappa)$ and $M^+_1(\operatorname{cl}_X(Sx),S) \neq \emptyset$ for every $x \in X$. Let $P\colon C(X) \to C(X)$ be the mean ergodic projection, so that $\ran(P) = \Fix(\kappa)$ and $\ker(P) = \wCob(\kappa)$. Fix $x \in X$ and $\mu \in M^+_1(\operatorname{cl}_X(Sx),S)$. 

For any $f \in C(X)$, we can write
\[
f = Pf + h
\]
with $h \in \wCob(\kappa)$. As $Pf \in \Fix(\kappa)$, we have that $(Pf)(sx) = (Pf)(x)$ for all $s \in S$, so $(Pf)(y) = (Pf)(x)$ for all $y \in \operatorname{cl}_X(Sx)$. Since $\Fix(\kappa^*) = \annil{\wCob(\kappa)}$, we have that
\[
\mu(f) = \mu(Pf) + \mu(h) = \mu(Pf) = \int Pf \,d\mu,
\]
and, as $\operatorname{supp}(\mu) \subseteq \operatorname{cl}_X(Sx)$,
\[
\mu(f) = \int Pf \, d\mu = \int_{\operatorname{cl}_X(Sx)}{Pf} \, d\mu = (Pf)(x).
\]
Therefore,
\[
\langle f,\mu\rangle = \langle Pf,\delta_x\rangle  = \langle f,P^*\delta_x\rangle. 
\]
Since $f$ was arbitrary, we conclude that $\mu = P^*\delta_x$ and $M^+_1(\operatorname{cl}_X(Sx),S) = \{P^*\delta_x\}$. As $x$ was arbitrary, we have that $S \acts X$ is pointwise uniquely ergodic.

It remains to prove that $\Phi(x) := P^*\delta_x$ is weak* continuous: If $x_n \to x$, then $\delta_{x_n} \to \delta_x$. Since $P^*$ is weak*-to-weak* continuous, $\Phi(x_n)=P^*\delta_{x_n} \to P^*\delta_x = \Phi(x)$. 

\medskip

Under either of the equivalent conditions, the mean ergodic projection $P\colon C(X) \to \Fix(\kappa)$ is defined by
\[
(Pf)(x) = \int f\,d\Phi(x)   \qquad   \text{for}~f \in C(X),~x \in X,
\]
and satisfies $P^*\delta_x = \Phi(x)$ for all $x \in X$. By Theorem~\ref{thm:psi-general}, this implies that the dual mean ergodic projection $Q\colon M(X) \to \Fix(\kappa^*)$ is given by
\[
Q\nu = \int \Phi \, d\nu \qquad   \text{for}~\nu \in M(X),
\]
where the integral is in the sense of Gelfand.
\end{proof}

\section{The locally compact case}
\label{sec:locally-compact-case}

Let $X$ be a separable locally compact metric space. If $X$ is non-compact, let $\widehat X=X \sqcup \{\infty\}$ be the one-point compactification of $X$. If $X$ is compact, we use the same notation for the compact space obtained by adjoining an isolated point $\infty$ to $X$. Under our assumptions, $\widehat X$ is metrizable. In either case, the topology induced on $X$ is the original one, and a neighborhood basis of $\infty$ is given by the sets $\widehat X\setminus K$, where $K\subseteq X$ is compact.

A sequence $(x_n)_n$ in $X$ is said to \textbf{converge to infinity}, and we write $x_n\to\infty$, if for every compact set $K\subseteq X$ there exists $n_0\in\mathbb N$ such that
\[
x_n\notin K \qquad \text{for every } n\geq n_0.
\]
Equivalently, this means that $x_n\to\infty$ in $\widehat X$. If $X$ is compact, no sequence in $X$ converges to infinity.

\subsection{Compactification of functions and measures}

Denote by
\[
I_\infty(\widehat X) := \{g\in C(\widehat X)\mid g(\infty)=0\}
\]
the maximal ideal of functions in $C(\widehat X)$ that vanish at $\infty$. A function $f\colon X\to\R$ belongs to $C_0(X)$ if and only if the extension $\overline{f}\colon \widehat X \to \R$ defined by
\[
\overline{f}(x):=
\begin{cases}
f(x)    &   \text{if}~x\in X\\
0       &   \text{if}~x=\infty
\end{cases}
\]
belongs to $C(\widehat X)$. Observe that a function $f\in C_b(X)$ belongs to $C_0(X)$ if and only if $\lim_n f(x_n)=0$ for every sequence $(x_n)_n$ in $X$ that converges to infinity or, equivalently, if $\lim_n \overline{f}(x_n)=0$ for every sequence $(x_n)_n$ in $\widehat{X}$ such that $x_n \to \infty$.

\begin{lemma}
\label{lem:vanishing-sequences}
Let $X$ be a separable locally compact metric space.
\begin{enumerate}
\item If $x_n\to\infty$, then $\wkslim_n \delta_{x_n} = 0$ in $M(X)$.
\item If $\Psi\colon X\to M(X)$ is weak* continuous, then $\Psi$ vanishes at infinity if and only if $\wkslim_n \Psi(x_n) = 0$ for every sequence $(x_n)_n$ in $X$ such that $x_n\to\infty$.
\end{enumerate}
\end{lemma}

\begin{proof}
For the first assertion, if $f\in C_0(X)$ and $x_n\to\infty$, then $\langle f,\delta_{x_n}\rangle=f(x_n)\to0$. Hence $\delta_{x_n}\to0$ weak* in $M(X)$. For the second assertion, if $\Psi$ vanishes at infinity and $f\in C_0(X)$,
then $f\Psi\in C_0(X)$, so
\[
\langle f,\Psi(x_n)\rangle=(f\Psi)(x_n)\to0
\]
whenever $x_n\to\infty$. Thus $\wkslim_n \Psi(x_n) = 0$. Conversely, assume $\wkslim_n\Psi(x_n) = 0$ for every sequence $x_n\to\infty$. Since $\Psi$ is weak* continuous, for each $f\in C_0(X)$ the scalar function $x\mapsto\langle f,\Psi(x)\rangle$ is continuous. The assumed sequential condition implies that this scalar function belongs to $C_0(X)$. Hence $\Psi$ vanishes at infinity.
\end{proof}

This yields an isometric isomorphism
\[
\varphi\colon C_0(X)\longrightarrow I_\infty(\widehat X), \qquad \varphi(f) = \overline{f},
\]
whose inverse is given by restriction to $X$. Thus, every $g\in C(\widehat X)$ decomposes uniquely as
\[
g=\left((g-g(\infty)\mathbbm{1}\right))+g(\infty)\mathbbm{1},
\]
and hence
\[
C(\widehat X) = I_\infty(\widehat X)\oplus \R\mathbbm{1} \cong C_0(X)\oplus \R\mathbbm{1},
\]
where $\mathbbm{1}$ denotes the function constantly equal to $1$ on
$\widehat X$. Similarly, let
\[
M(\widehat X,\{\infty\}) := \{\nu\in M(\widehat X)\mid \nu(\{\infty\})=0\},
\]
and consider the canonical embedding $j\colon X\hookrightarrow \widehat X$. We identify $M(X)$ with the subspace $M(\widehat X,\{\infty\})$ via the pushforward $j_*$, that is, by viewing measures on $X$ as measures on $\widehat{X}$ with no mass at $\infty$. Hence every $\nu\in M(\widehat X)$ decomposes uniquely as
\[
\nu = \left((\nu-\nu(\{\infty\})\delta_\infty\right)) + \nu(\{\infty\})\delta_\infty,
\]
and therefore
\[
M(\widehat X) = M(\widehat X,\{\infty\})\oplus \R\delta_\infty \cong M(X)\oplus \R\delta_\infty.
\]
At the level of probability measures, restriction to $X$ gives an affine weak*-to-weak* homeomorphism
\[
M_1^+(\widehat X)\longrightarrow M_{\leq 1}^+(X), \qquad \nu\longmapsto \nu|_X,
\]
whose inverse is
\[
\iota\colon M_{\leq 1}^+(X)\longrightarrow M_1^+(\widehat X), \qquad \iota(\eta):=j_*\eta+(1-\eta(X))\delta_\infty.
\]
Thus $\iota$ extends the embedding $j$ at the level of Dirac measures, that is, $\iota(\delta_x)=\delta_{j(x)}$ for $x\in X$, and sends the zero measure to the mass concentrated at infinity: $\iota(0)=\delta_\infty$.

\subsection{Compactification of the action and Koopman representations}

For a locally compact metric space $Y$, a continuous map $F\colon X \to Y$ is proper if and only if it admits a continuous extension $\widehat{F}\colon \widehat{X} \to \widehat{Y}$ such that $\widehat{F}|_X=F$ and $\widehat{F}(\infty_X)=\infty_Y$ (see \cite[Ch. I \S~10.3]{bourbaki1989}). In consequence, an action $S \acts X$ is an action by proper maps if and only if it extends to an action by continuous maps $S \acts \widehat{X}$ that fixes $\infty$. 

Let $S\acts X$ be an action by proper maps, and let $S\acts \widehat X$ be its (unique) extension fixing $\infty$. Consider the affine weak*-to-weak* homeomorphism
\[
M^+_1(\widehat{X},S)\longrightarrow M^+_{\leq1}(X,S), \qquad \nu\longmapsto \nu|_X.
\]
By the compact case recalled above, $M^+_1(\widehat X,S)$ is a Choquet simplex. Therefore, via this affine weak*-to-weak* homeomorphism, $M^+_{\leq1}(X,S)$ is also a Choquet simplex. Its extreme points are exactly the $S$-ergodic invariant subprobability measures, denoted by $E^+_{\leq1}(X,S)$. Moreover,
\[
M^+_{\leq1}(X,S)=\conv(M^+_1(X,S)\cup\{0\})
\qquad \text{and} \qquad E^+_{\leq1}(X,S)=E^+_1(X,S)\cup\{0\}.
\]

Let
\[
\kappa\colon S^{\operatorname{op}}\longrightarrow \mathcal B(C_0(X)) \qquad \text{and} \qquad \kappa_\infty\colon S^{\operatorname{op}}\longrightarrow \mathcal B(C(\widehat{X}))
\]
be the corresponding Koopman representations. Since $\widehat{s}(\infty)=\infty$ for every $s\in S$, the ideal $I_\infty(\widehat X)$ is $\kappa_\infty$-invariant. Therefore, through the isometric isomorphism $\varphi$, we obtain that
\[
\kappa(s) = \varphi^{-1} \circ \kappa_\infty(s) \circ \varphi
\qquad \text{for every}~s\in S.
\]
For any $f \in C_0(X)$, denote $f_\infty := \varphi(f)$. Notice that $f_\infty \in I_\infty(\widehat X)$.

\begin{lemma}
\label{lem:decomposition}
Let $S\acts X$ be an action by proper maps, and let $S\acts \widehat X$ be its (unique) extension fixing $\infty$. Then,
\[
C_0(X)=\Fix(\kappa)\oplus\wCob(\kappa)  \qquad \text{if and only if} \qquad C(\widehat{X})=\Fix(\kappa_\infty)\oplus\wCob(\kappa_\infty).
\]
\end{lemma}

\begin{proof}
Since $I_\infty(\widehat X)$ is a closed $\kappa_\infty$-invariant subspace of $C(\widehat{X})$, we can consider the restricted representation $\kappa_\infty^0\colon S^{\operatorname{op}} \to \mathcal{B}(I_\infty(\widehat{X}))$. The isometric isomorphism $\varphi\colon C_0(X)\to I_\infty(\widehat X)$ intertwines $\kappa$ with $\kappa_\infty^0$, that is,
\[
\varphi\circ\kappa(s)=\kappa_\infty^0(s)\circ\varphi
\qquad \text{for every}~s\in S.
\]
Therefore, for any $f \in C_0(X)$ and $s \in S$, we have that
\[
f_\infty = \kappa_\infty^0(s)f_\infty \qquad \iff \qquad f = \kappa(s) f,
\]
where $f_\infty = \varphi(f)$, and
\[
f_\infty - \kappa_\infty^0(s)f_\infty = \varphi(f) - \kappa_\infty^0(s)\varphi(f) = \varphi(f) - \varphi(\kappa(s)f) = \varphi(f - \kappa(s)f).
\]
Thus,
\[
\Fix(\kappa_\infty^0) = \varphi(\Fix(\kappa))
\qquad \text{and} \qquad \wCob(\kappa_\infty^0) = \varphi(\wCob(\kappa)).
\]

Since $C(\widehat{X})=I_\infty(\widehat{X})\oplus\R\mathbbm{1}$ and $\kappa_\infty(s)\mathbbm{1}=\mathbbm{1}$, we have
\[
\Fix(\kappa_\infty)=\Fix(\kappa_\infty^0)\oplus\R\mathbbm{1} \qquad \text{and} \qquad \wCob(\kappa_\infty)=\wCob(\kappa_\infty^0).
\]
Therefore,
\begin{align*}
C_0(X)=\Fix(\kappa)\oplus\wCob(\kappa)  &   \iff I_\infty(\widehat{X})=\Fix(\kappa_\infty^0)\oplus\wCob(\kappa_\infty^0)  \\
                                            &   \iff C(\widehat{X})=\Fix(\kappa_\infty)\oplus\wCob(\kappa_\infty).
\end{align*}
\end{proof}

Under the direct sum identification
\[
C(\widehat X) = I_\infty(\widehat X)\oplus \R\mathbbm{1} \cong C_0(X)\oplus \R\mathbbm{1},
\]
the Koopman representation $\kappa_\infty$ is the unitization of $\kappa$. Indeed, for every $s\in S$, $f\in C_0(X)$, and $\alpha\in\R$,
\[
\kappa_\infty(s)(\varphi(f)+\alpha\mathbbm{1}) = \varphi(\kappa(s)f)+\alpha\mathbbm{1}.
\]
In particular, the constant functions are fixed and $\kappa_\infty(s)\mathbbm{1}=\mathbbm{1}$ for every $s\in S$.

For a subset $A\subseteq X$, write $\operatorname{cl}_X(A)$ for its closure inside $X$, and $\operatorname{cl}_{\widehat X}(A)$ for its closure inside $\widehat X$. If $S\acts X$ is an action by proper maps, then for every $x\in X$,
\[
\operatorname{cl}_{\widehat X}(Sx) = \operatorname{cl}_X(Sx)
\]
whenever $\operatorname{cl}_X(Sx)$ is compact, while
\[
\operatorname{cl}_{\widehat X}(Sx) = \operatorname{cl}_X(Sx)\cup\{\infty\}
\]
whenever $\operatorname{cl}_X(Sx)$ is not compact, that is, when $x$ escapes to infinity.

\subsection{Pointwise ergodicity under compactification}
\label{subsection:pointwise-ergodicity}

An action by proper maps $S\acts X$ is pointwise uniquely ergodic if and only if the compactified action $S\acts \widehat X$ is pointwise uniquely ergodic. Indeed, if $\operatorname{cl}_X(Sx)$ is compact, then the invariant probability measures on $\operatorname{cl}_X(Sx)$ and on $\operatorname{cl}_{\widehat X}(Sx)$ coincide. If $\operatorname{cl}_X(Sx)$ is not compact, then
\[
\operatorname{cl}_{\widehat X}(Sx)=\operatorname{cl}_X(Sx)\cup\{\infty\}.
\]
Since $\infty$ is fixed, restriction to $X$ identifies $M_1^+(\operatorname{cl}_{\widehat X}(Sx),S)$ with $M_{\leq 1}^+(\operatorname{cl}_X(Sx),S)$, and $\delta_\infty$ corresponds to the zero subprobability. Hence $\delta_\infty$ is the unique invariant probability measure on $\operatorname{cl}_{\widehat X}(Sx)$ if and only if there is no non-zero $S$-invariant subprobability measure on $\operatorname{cl}_X(Sx)$, equivalently no $S$-invariant probability measure there.

Notice that a map $\Psi\colon X \to M^+_{\leq 1}(X)$ is weak* continuous and vanishes at infinity if and only if the map $\Psi_\infty\colon \widehat{X}\to M^+_1(\widehat{X})$ defined by
\[
\Psi_\infty(x) = 
\begin{cases}
\iota(\Psi(x))
&   \text{if}~x \in X \\
\delta_\infty   &   \text{if}~x=\infty
\end{cases}
\]
is weak* continuous. Indeed, let $g\in C(\widehat X)$ and write
\[
g=f_\infty+g(\infty)\mathbbm{1}, \qquad f=(g-g(\infty)\mathbbm{1})|_X\in C_0(X).
\]
For $x\in X$,
\[
\int_{\widehat X}g\,d\Psi_\infty(x) = \int_X f\,d\Psi(x)+g(\infty),
\]
while at $x=\infty$ this integral equals $g(\infty)$. Thus continuity of $\Psi_\infty$ at points of $X$ is equivalent to weak* continuity of $\Psi$, and continuity at $\infty$ is equivalent to $x\longmapsto f\Psi(x) := \int f\, d\Psi(x)$ vanishing at infinity for every $f\in C_0(X)$.

In particular, a pointwise uniquely ergodic action by proper maps $S\acts X$ is continuously pointwise ergodic and vanishes at infinity if and only if the map $\Phi_\infty\colon \widehat{X}\to M_1^+(\widehat{X},S)$ is weak* continuous, where $\Phi\colon X\to M_{\leq 1}^+(X)$ is the map associated to $S\acts X$. Observe that $\Phi_\infty$ is defined by
\[
\Phi_\infty(x) = 
\begin{cases}
j_*\Phi(x)      &   \text{if}~x \in X~\text{and}~\operatorname{cl}_X(Sx)~\text{is compact}\\
\delta_\infty   &   \text{otherwise}.
\end{cases}
\]
When the compactified action is pointwise uniquely ergodic, then the associated map $\widehat{X}\to M_1^+(\widehat{X},S)$ is $\Phi_\infty$.

\begin{theorem}
\label{thm:locally-compact-main}
Let $S$ be a semigroup, $X$ a separable locally compact metric space, and $S\acts X$ a pointwise invariant-measure admitting action by proper maps. Let $S\acts \widehat X$ be the
extension fixing $\infty$. Then the following are equivalent:
\begin{enumerate}
\item $S\acts X$ is continuously pointwise ergodic and vanishes at infinity.
\item $S\acts \widehat X$ is continuously pointwise ergodic.
\item $C_0(X)=\Fix(\kappa)\oplus\wCob(\kappa)$.
\item $M(X)=\Fix(\kappa^*)\oplus\wsCob(\kappa^*)$.
\end{enumerate}
If any of these equivalent conditions holds, then the mean ergodic projection $P\colon C_0(X)\to\Fix(\kappa)$ is given by
\[
(Pf)(x)=\int_X f\,d\Phi(x),
\qquad f\in C_0(X),\ x\in X.
\]
Equivalently, if $P_\infty$ denotes the mean ergodic projection of the compactified action, then
\[
P_\infty(f_\infty+\alpha\mathbbm{1}) = (Pf)_\infty+\alpha\mathbbm{1}.
\]
Moreover, under the identification of $M(X)$ with measures on $\widehat X$ with no mass at $\infty$, the dual projection satisfies
\[
\langle f,P^*\nu\rangle = \int_X\left(\int_X f\,d\Phi(x)\right)d\nu(x), \qquad f\in C_0(X),\ \nu\in M(X).
\]
\end{theorem}

\begin{proof}
The equivalence between (3) and (4) follows from Proposition~\ref{prop:ann-cob-koopman}. We prove $(1)\iff(2)$. We already established that $S \acts X$ is pointwise uniquely ergodic if and only if $S \acts \widehat{X}$ is pointwise uniquely ergodic. Thus it suffices to prove that $\Phi\colon X \to M_{\leq 1}^+(X)$ is weak* continuous and vanishes at infinity if and only if $\Phi_\infty\colon \widehat{X} \to M_1^+(\widehat{X},S)$ is weak* continuous under the hypothesis that $S \acts \widehat{X}$ is pointwise uniquely ergodic. Recalling that $\Phi_\infty$ is weak* continuous if and only if $\Phi$ is weak* continuous and vanishes at infinity, the result follows.

We prove $(2)\iff(3)$. The assumption that $S\acts X$ is pointwise invariant-measure admitting implies that the compactified action $S\acts\widehat X$ is also pointwise invariant-measure admitting: for $x=\infty$ the measure is $\delta_\infty$, for points $x\in X$ escaping to infinity, again $\delta_\infty$ belongs to the compactified orbit closure, and for non-escaping points the invariant probability measures are the same as in $X$. By Theorem~\ref{thm:main1}, as $\widehat{X}$ is compact and metrizable, under the assumption that $M_1^+(\operatorname{cl}_{\widehat{X}}(Sx),S) \neq \emptyset$ for every $x \in \widehat{X}$, $S \acts \widehat{X}$ is continuously pointwise ergodic if and only if $C(\widehat{X})=\Fix(\kappa_\infty)\oplus\wCob(\kappa_\infty)$. By Lemma~\ref{lem:decomposition}, $C_0(X)=\Fix(\kappa)\oplus\wCob(\kappa)$ if and only if $C(\widehat{X})=\Fix(\kappa_\infty)\oplus\wCob(\kappa_\infty)$. Thus it suffices to prove that $M_1^+(\operatorname{cl}_{\widehat{X}}(Sx),S) \neq \emptyset$ for every $x \in \widehat{X}$ if and only if $M_1^+(\operatorname{cl}_{X}(Sx),S) \neq \emptyset$ for every $x \in X$ such that $\operatorname{cl}_{X}(Sx)$ is compact. This is direct, as we already observed that if $\operatorname{cl}_X(Sx)$ is compact, then the invariant probability measures on $\operatorname{cl}_X(Sx)$ and on $\operatorname{cl}_{\widehat X}(Sx)$ coincide, and if
$\operatorname{cl}_X(Sx)$ is not compact, then $\delta_\infty\in M^+_1(\operatorname{cl}_{\widehat X}(Sx),S)$.

Under these equivalent conditions, the compact theorem applied to $S\acts\widehat X$ gives
\[
(P_\infty g)(z)=\int_{\widehat X}g\,d\Phi_\infty(z), \qquad g\in C(\widehat X),\ z\in\widehat X.
\]
Taking $g=f_\infty$ and $z=x\in X$, we obtain
\[
(Pf)(x)=\int_X f\,d\Phi(x).
\]
The formula
\[
P_\infty(f_\infty+\alpha\mathbbm{1}) = (Pf)_\infty+\alpha\mathbbm{1}
\]
then follows from linearity and from $P_\infty\mathbbm{1}=\mathbbm{1}$. Finally, for $f\in C_0(X)$ and $\nu\in M(X)$,
\[
\langle f,P^*\nu\rangle = \langle Pf,\nu\rangle = \int_X\left(\int_X f\,d\Phi(x)\right)d\nu(x).
\]
\end{proof}

\section{Uniform systems and amenable semigroups}
\label{section5}

In this section, we establish a characterization of mean ergodicity in terms of convergence of averages.

\subsection{Amenable semigroups}

Let $(S,\cdot)$ be a countable and discrete semigroup, and let $\FSet$ denote the set of all non-empty finite subsets of $S$. The semigroup $S$ is \textbf{bicancellative} if, for all $s,t_1,t_2 \in S$,
\[
\left(t_1s = t_2s \qquad \text{or} \qquad st_1 = st_2\right)   \qquad  \text{implies} \qquad t_1 = t_2.
\]
If $S$ is bicancellative, a sequence $(F_n)_n$ in $\FSet$ is \textbf{right F{\o}lner} (resp. \textbf{left F{\o}lner}) if
\[
\lim_n \frac{|F_n \Delta F_ns|}{|F_n|}=0 \quad \left(\text{resp.} \quad \lim_n \frac{|sF_n \Delta F_n|}{|F_n|}=0 \right)   \quad   \text{for every}~s\in S.
\]
If $S$ is bicancellative, it is said to be \textbf{right amenable} (resp. \textbf{left amenable}) if there exists a right (resp. left) F{\o}lner sequence. Clearly, if $S$ is right amenable (resp. left amenable), then its opposite semigroup $S^{\operatorname{op}}$ is left amenable (resp. right amenable).

\begin{lemma}
\label{lem:folner-stability}
Let $S$ be bicancellative and let $(F_n)_n$ be a left F{\o}lner sequence. Then, for every $s\in S$, both $(sF_n)_n$ and $(F_ns)_n$ are left F{\o}lner sequences. Moreover, every subsequence of a left F{\o}lner sequence is left F{\o}lner, and the sequence obtained by alternating two left F{\o}lner sequences is again left F{\o}lner.
\end{lemma}

\begin{proof}
The assertions about subsequences and alternating sequences are immediate.

For left translates, fix $s,t\in S$. Since left multiplication by $s$ is injective, $|sF_n|=|F_n|$, and
\[
|t(sF_n)\Delta sF_n| = |(ts)F_n\Delta sF_n| \leq |(ts)F_n\Delta F_n|+|F_n\Delta sF_n| = o(|F_n|).
\]
Thus $(sF_n)_n$ is left F{\o}lner. For right translates, right cancellativity gives $|F_ns|=|F_n|$, and
\[
|t(F_ns)\Delta F_ns| = |(tF_n)s\Delta F_ns| = |(tF_n\Delta F_n)s| = |tF_n\Delta F_n| = o(|F_n|).
\]
Thus $(F_ns)_n$ is left F{\o}lner.
\end{proof}

\subsection{Averages}

Assume that $S$ is a countable, discrete, bicancellative, and left amenable semigroup, $X$ is a separable locally compact metric space, and $S \acts X$ is an action by proper maps. Given $F \in \FSet$, define the \textbf{$F$-average operator} and the \textbf{dual $F$-average operator} by
\[
A_F := \frac{1}{|F|} \sum_{t \in F} \kappa(t) \qquad \text{and} \qquad A^*_F := \frac{1}{|F|} \sum_{t \in F} \kappa^*(t),
\]
respectively. Note that, as $\sum_{t \in F} \frac{1}{|F|} = 1$, $A_F$ belongs to $\conv(\kappa(S))$. Moreover, for any $s \in S$,
\[
A_F  \kappa(s) = A_{sF} \qquad \text{and} \qquad \kappa(s) A_F = A_{Fs},
\]
and, similarly,
\[
A^*_F  \kappa^*(s) = A^*_{Fs} \qquad \text{and} \qquad \kappa^*(s) A^*_F = A^*_{sF}.
\]

In this setting, an action by proper maps $S\acts X$ is called
\begin{itemize}
\item \textbf{uniform} if for every left F{\o}lner sequence $(F_n)_n$, the map
\[
f \in C_0(X) \longmapsto \lim_n A_{F_n} f,
\]
is well-defined in the uniform norm topology, and
\item \textbf{weak* mean continuous} if for every left F{\o}lner sequence $(F_n)_n$, the map
\[
\nu \in M^+_{\leq 1}(X) \longmapsto \wkslim_n A^*_{F_n} \nu,
\]
is well-defined and weak*-to-weak* continuous.
\end{itemize}

The definition of uniformity above extends the classical compact notion, since $C_0(X)=C(X)$ when $X$ is compact. To the best of our knowledge, weak* mean continuity in the form introduced here has not previously been studied.

The following result is a version of the Mean Ergodic Theorem due to Eberlein \cite[Theorem~3.1]{eberlein1949}. We follow  \cite[\S~2, Theorem~1.5]{krengel1985}.

\begin{theorem}
\label{thm:mean}
Let $S$ be a countable, discrete, bicancellative, and left amenable semigroup, $X$ a separable locally compact metric space, and $S \acts X$ an action by proper maps. Fix a left F{\o}lner sequence $(F_n)_n$. For $f \in C_0(X)$ and $g \in \Fix(\kappa)$, the following are equivalent:
\begin{enumerate}
\item $g = \lim_n A_{F_n} f$.
\item $g(x) = \lim_n A_{F_n} f(x)$ for every $x \in X$.
\end{enumerate}
As a consequence, the following are also equivalent:
\begin{enumerate}
\item[(a)] $f \in \wCob(\kappa)$.
\item[(b)] $\lim_n A_{F_n} f = 0$.
\item[(c)] $\lim_n A_{F_n} f(x) = 0$ for every $x \in X$.
\end{enumerate}
\end{theorem}

\begin{proof}
The implication $(1) \implies (2)$ is clear. To prove $(2) \implies (1)$, assume that $g(x) = \lim_n A_{F_n} f(x)$ for every $x \in X$. Since $\|A_{F_n}f\|_\infty \leq \|f\|_\infty$ for every $n$, for every $\nu \in M(X)$, by the Dominated Convergence Theorem applied to $|\nu|$,
\[
\lim_n \langle A_{F_n}f,\nu\rangle = \lim_n \int{A_{F_n}f(x)}\,d\nu(x) = \int{g(x)}\,d\nu(x) = \langle g,\nu\rangle,
\]
showing that the weak limit $\wklim_n A_{F_n}f = g$ holds. Observe that $\{A_{F_n} f\}_n \subseteq \conv\{\kappa(s)f \mid s \in S\}$. Thus, $g \in \cconv\{\kappa(s)f \mid s \in S\}$, since by
Mazur's Theorem~\cite[Theorem~3.12]{rudin1991}, the norm and weak closure of a convex set in a Banach space coincide. Fix $\epsilon > 0$. The condition $g \in \cconv(\kappa(S)f)$ implies the existence of a finite convex combination $g_\epsilon = \sum_{j=1}^m \lambda_j \kappa(s_j)f$, with $\lambda_j > 0$ and $\sum_{j=1}^m \lambda_j = 1$, such that $\|g - g_\epsilon\| < \tfrac{\epsilon}{2}$. Notice that
\begin{align*}
\|A_{F_n} g_\epsilon - A_{F_n} f\|_\infty = \left\|\sum_{j=1}^m \lambda_j (A_{F_n} \kappa(s_j)f - A_{F_n} f)\right\|_\infty \leq \sum_{j=1}^m \lambda_j \|A_{F_n} \kappa(s_j)f - A_{F_n} f\|_\infty \\
= \sum_{j=1}^m \lambda_j\left\|A_{s_j F_n}f - A_{F_n}f\right\|_\infty  \leq    \sum_{j=1}^m \lambda_j\frac{\left|s_j F_n \Delta F_n\right|}{|F_n|}\|f\|_\infty.
\end{align*}
By left amenability, we can find $n_0$ such that $\sum_{j=1}^m \lambda_j\frac{\left|s_j F_n \Delta F_n\right|}{|F_n|}\|f\|_\infty < \frac{\epsilon}{2}$ for every $n \geq n_0$. As $g$ belongs to $\Fix(\kappa)$, $A_{F_n} g = g$ for all $n$. We therefore obtain
\[
\|g - A_{F_n} f\| \leq \|A_{F_n}(g - g_\epsilon)\| + \|A_{F_n} g_\epsilon - A_{F_n} f\| \leq \|g - g_\epsilon\| + \frac{\epsilon}{2} < \frac{\epsilon}{2} + \frac{\epsilon}{2} = \epsilon.
\]
Since $\epsilon$ was arbitrary, this shows that $\lim_n A_{F_n} f = g$ in the norm topology.

It remains to show that $(a) \iff (b)$. The equivalence $(b) \iff (c)$ is a particular case of $(1) \iff (2)$ applied to $g \equiv 0$.

Assume that $f \in \wCob(\kappa)$. Then, for any $\epsilon > 0$, there exists $\lambda_1,\dots,\lambda_m \in \R$, $s_1,\dots,s_m \in S$, $f_1,\dots,f_m \in C_0(X)$ such that
\[
\left\|f - \sum_{j=1}^m\lambda_j(f_j - \kappa(s_j)f_j)\right\|_\infty < \epsilon.
\]
Thus, for any left F{\o}lner sequence $(F_n)_n$,
\begin{align*}
\left\|A_{F_n}f\right\|_\infty \leq   \left\|A_{F_n}\sum_{j=1}^m\lambda_j(f_j - \kappa(s_j)f_j)\right\|_\infty + \epsilon  \leq   \sum_{j=1}^m|\lambda_j|\left\|A_{F_n}f_j - A_{F_n}\kappa(s_j)f_j\right\|_\infty + \epsilon  \\
                                    \leq    \sum_{j=1}^m|\lambda_j|\left\|A_{F_n}f_j - A_{s_j F_n}f_j\right\|_\infty + \epsilon  \leq    \sum_{j=1}^m|\lambda_j|\frac{\left|s_j F_n \Delta F_n\right|}{|F_n|}\|f_j\|_\infty + \epsilon,
\end{align*}
so
\[
\limsup_n \left\|A_{F_n}f\right\|_\infty \leq  \sum_{j=1}^m|\lambda_j|\lim_n\frac{\left|s_j F_n \Delta F_n\right|}{|F_n|}\|f_j\|_\infty + \epsilon = \epsilon,
\]
and since $\epsilon$ was arbitrary, we conclude that $\lim_n A_{F_n}f = 0$. Conversely, assume $\lim_n A_{F_n}f = 0$. Then, for any $n$,
\[
f-A_{F_n}f = \frac{1}{|F_n|} \sum_{t \in F_n} (f-\kappa(t)f) \in \Cob(\kappa),
\]
so
\[
f = f - \lim_n A_{F_n}f = \lim_n(f-A_{F_n}f) = \lim_n \frac{1}{|F_n|} \sum_{t \in F_n} (f-\kappa(t)f) \in \wCob(\kappa).
\]
\end{proof}

Next, we prove a version of the Krylov–Bogolyubov theorem adapted to our setting.

\begin{lemma}
\label{lem:krylov}
Let $S$ be a countable, discrete, bicancellative, and left amenable semigroup, $X$ a separable locally compact metric space, and $S \acts X$ an action by proper maps. Then, $S \acts X$ is pointwise invariant-measure admitting, that is, $M_1^+(\operatorname{cl}_X(Sx),S) \neq \emptyset$ for every $x \in X$ such that $\operatorname{cl}_X(Sx)$ is compact.
\end{lemma}

\begin{proof}
Fix $x \in X$ such that $\operatorname{cl}_X(Sx)$ is compact, and fix a left F{\o}lner sequence $(F_n)_n$. Since $\operatorname{cl}_X(Sx)$ is compact and metrizable, $M_1^+(\operatorname{cl}_X(Sx))$ is weak* compact and metrizable. Hence there exists a subsequence $(F_{n_k})_k$ and a measure $\mu\in M_1^+(\operatorname{cl}_X(Sx))$ such that
\[
\wkslim_k A_{F_{n_k}}^*\delta_x = \mu.
\]
We claim that $\mu$ is $S$-invariant. Fix $s\in S$ and $f\in C_0(X)$. Then
\begin{align*}
\left|\left\langle f,\kappa^*(s)\mu-\mu\right\rangle\right| & =
\lim_k\left|
\left\langle f,\kappa^*(s)A_{F_{n_k}}^*\delta_x-A_{F_{n_k}}^*\delta_x\right\rangle
\right|  \\
    &   = \lim_k\left|\left\langle A_{sF_{n_k}}f-A_{F_{n_k}}f,\delta_x\right\rangle\right|   \leq    \lim_k  \frac{|sF_{n_k}\Delta F_{n_k}|}{|F_{n_k}|}\|f\|_\infty=0.
\end{align*}
Thus $\kappa^*(s)\mu=\mu$ for every $s\in S$, and hence $\mu\in M_1^+(\operatorname{cl}_X(Sx),S)$.
\end{proof}

Note that, for general semigroups $S$, even if $\operatorname{cl}_X(Sx)$ is compact, the existence of an $S$-invariant measure in $M_1^+(\operatorname{cl}_X(Sx),S)$ is not guaranteed.

\begin{theorem}
\label{thm:main2}
Let $S$ be a countable, discrete, bicancellative, and left amenable semigroup, $X$ a separable locally compact metric space, and $S \acts X$ an action by proper maps. The following are equivalent:
\begin{enumerate}
    \item $S \acts X$ is continuously pointwise ergodic and vanishes at infinity.
    \item $C_0(X) = \Fix(\kappa) \oplus \wCob(\kappa)$.
    \item $S \acts X$ is uniform.
    \item $M(X) = \Fix(\kappa^*) \oplus \wsCob(\kappa^*)$.
    \item $S \acts X$ is weak* mean continuous.
\end{enumerate}
Fix any left F{\o}lner sequence $(F_n)_n$. If any of these holds, then:
\begin{itemize}
\item The mean ergodic projection $P\colon C_0(X) \to \Fix(\kappa)$ is defined by
\[
Pf = \lim_n A_{F_n}f   \qquad   \text{for}~f \in C_0(X).
\]
\item The dual mean ergodic projection $P^*\colon M(X) \to \Fix(\kappa^*)$ is defined by
\[
P^*\nu = \wkslim_n A^*_{F_n}\nu \qquad   \text{for}~\nu \in M(X).
\]
\end{itemize}
\end{theorem}

\begin{proof}
We establish $(1) \iff (2)$, $(2) \iff (3)$, $(2) \iff (4)$, and $(3) \iff (5)$, which suffices to prove the equivalence of all assertions.

\medskip
\noindent
\textbf{$(1) \iff (2)$:} By Lemma~\ref{lem:krylov}, the condition $M^+_1(\operatorname{cl}_X(Sx),S) \neq \emptyset$ for every $x \in X$ such that $\operatorname{cl}_X(Sx)$ is compact is automatically fulfilled, and the equivalence between (1) and (2) follows from Theorem~\ref{thm:locally-compact-main}.

\medskip
\noindent
\textbf{$(2) \implies (3)$:} Assume that $C_0(X) = \Fix(\kappa) \oplus \wCob(\kappa)$, and let $P\colon C_0(X) \to \Fix(\kappa)$ be the mean ergodic projection. Then, for any $f \in C_0(X)$, we can write
\[
f = Pf + g,
\]
with $Pf \in \Fix(\kappa)$ and $g \in \wCob(\kappa)$. Therefore, for any left F{\o}lner sequence $(F_n)_n$, we have
\[
\lim_n A_{F_n} f = \lim_n A_{F_n}Pf + \lim_n A_{F_n}g = Pf,
\]
since $A_{F_n}Pf = Pf$ and, due to Theorem~\ref{thm:mean}, $\lim_n A_{F_n}g = 0$. As $f$ was arbitrary, we have that $S \acts X$ is uniform and $Pf = \lim_n A_{F_n}f$ for any $f \in C_0(X)$.

\medskip
\noindent
\textbf{$(3) \implies (2)$:} Given a left F{\o}lner sequence $(F_n)_n$, define
\[
\widetilde{P}\colon C_0(X) \longrightarrow C_0(X), \qquad f \longmapsto \widetilde{P}f := \lim_n A_{F_n}f.
\]
Clearly, $\widetilde{P}$ is a linear and bounded operator with $\|\widetilde{P}\| \leq 1$. In addition, $\widetilde{P}$ does not depend on the left F{\o}lner sequence as, by Lemma~\ref{lem:folner-stability}, a sequence obtained by alternating the elements of two given left F{\o}lner sequences is again a left F{\o}lner sequence. 

First, note that $\widetilde{P}f = f$ for every $f \in \Fix(\kappa)$. Second, for any $s \in S$,
\[
\kappa(s)\widetilde{P}f = \kappa(s) \lim_n A_{F_n}f = \lim_n \kappa(s)A_{F_n}f = \lim_n A_{F_ns}f = \widetilde{P}f,
\]
as $(F_n s)_n$ is also a left F{\o}lner sequence. Thus, $\ran(\widetilde{P}) \subseteq \Fix(\kappa)$. Combining these two assertions we obtain that $\ran(\widetilde{P}) = \Fix(\kappa)$ and $\widetilde{P}^2 = \widetilde{P}$. That $\ker(\widetilde{P}) = \wCob(\kappa)$ follows directly from Theorem~\ref{thm:mean}, as $\lim_n A_{F_n} f = 0$ if and only if $f \in \wCob(\kappa)$. Therefore, $\widetilde{P}$ is the mean ergodic projection $P$, and
\[
C_0(X) = \ran(\widetilde{P}) \oplus \ker(\widetilde{P}) = \Fix(\kappa) \oplus \wCob(\kappa).
\]

\medskip
\noindent
\textbf{$(2) \iff (4)$:} This follows directly from Proposition~\ref{prop:ann-cob-koopman}.

\medskip
\noindent
\textbf{$(3) \implies (5)$:} Assume that $S \acts X$ is uniform. Since $(3)$ implies $(2)$ and $(4)$, $C_0(X) = \Fix(\kappa) \oplus \wCob(\kappa)$ and $M(X) = \Fix(\kappa^*) \oplus \wsCob(\kappa^*)$. Let $P\colon C_0(X) \to \Fix(\kappa)$ be the mean ergodic projection, and let $P^*\colon M(X) \to \Fix(\kappa^*)$ be its adjoint, equivalently the
dual mean ergodic projection. Let $(F_n)_n$ be a left F{\o}lner sequence. We claim that
\[
P^*\nu = \wkslim_n A^*_{F_n} \nu \qquad \text{for } \nu \in M(X).
\]
Indeed, for any $f \in C_0(X)$, $Pf = \lim_n A_{F_n}f = \wklim_n A_{F_n}f$, so
\[
\lim_n \langle f,A^*_{F_n}\nu\rangle = \lim_n \langle A_{F_n}f,\nu\rangle = \langle Pf,\nu\rangle = \langle f,P^*\nu\rangle.
\]
Therefore, the map
\[
\nu \in M_{\leq 1}^+(X) \longmapsto \wkslim_n A^*_{F_n}\nu
\]
is exactly the restriction $\left.P^*\right\vert_{M_{\leq 1}^+(X)}$. As $P^*$ is weak*-to-weak* continuous, the result follows.

\medskip
\noindent
\textbf{$(5) \implies (3)$:} Assume that $S \acts X$ is weak* mean continuous, and let
\[
q\colon M_{\leq 1}^+(X)\longrightarrow M(X), \qquad \nu \longmapsto q(\nu) :=\wkslim_n A_{F_n}^*\nu
\]
for some fixed left F{\o}lner sequence $(F_n)_n$. Observe that $q$ does not depend on the left F{\o}lner sequence as a sequence obtained by alternating the elements of two given left F{\o}lner sequences is again a left F{\o}lner sequence (see Lemma~\ref{lem:folner-stability}). Define
\[
\Psi\colon X \longrightarrow M(X), \qquad x \longmapsto \Psi(x) := \wkslim_n A_{F_n}^*\delta_x.
\]
Then $\Psi=q\circ \iota_X$, and as $q$ is weak*-to-weak* continuous and $\iota_X$ is weak* continuous, $\Psi$ is weak* continuous as well. Moreover, if $x_n\to\infty$ in $X$, then
$\delta_{x_n}\to 0$ weak* in $M_{\leq 1}^+(X)$, and since
\[
q(0)=\wkslim_n A_{F_n}^*0=0,
\]
the weak*-to-weak* continuity of $q$ gives
\[
\Psi(x_n)=q(\delta_{x_n})\to q(0)=0.
\]
Hence $\Psi$ vanishes at infinity. By Theorem~\ref{thm:psi-general}, there exists a unique bounded linear operator $P\colon C_0(X) \to C_0(X)$ such that $P^* \circ \iota_X = \Psi$, which is defined by
\[
Pf(x) = \int{f}d\Psi(x).
\]
Fix $f \in C_0(X)$. Then, for any $x \in X$,
\[
\lim_n A_{F_n}f(x) = \lim_n \langle A_{F_n}f,\delta_x\rangle = \lim_n \langle f,A^*_{F_n}\delta_x\rangle = \langle f,\Psi(x)\rangle = Pf(x).
\]
Since, for any $s \in S$, $(F_n s)_n$ is also a left F{\o}lner sequence and
\[
Pf(x) = \lim_n A_{F_n s}f(x) = \lim_n (\kappa(s) A_{F_n}f)(x) = \lim_n (A_{F_n}f)(sx) = Pf(sx) = \kappa(s)Pf(x).
\]
Therefore, $Pf$ belongs to $\Fix(\kappa)$ and, by Theorem~\ref{thm:mean}, we have $\lim_n A_{F_n}f = Pf$. Since the left F{\o}lner sequence $(F_n)_n$ and the function $f$ were arbitrary, $S \acts X$ is uniform.
\end{proof}

\begin{remark}
A direct application of Theorem~\ref{thm:mean} shows that the following seemingly weaker property is equivalent to $S \acts X$ being uniform: there exists a left F{\o}lner sequence $(F_n)_n$ such that, for every $f \in C_0(X)$, there is $g \in \Fix(\kappa)$ for which the following pointwise convergence holds:
\[
\lim_n A_{F_n}f(x) = g(x) \qquad \text{for every}~x \in X.
\]
\end{remark}

\begin{remark}
If $X$ is compact, in order to obtain weak* mean continuity, it suffices to check that, for every left F{\o}lner sequence $(F_n)_n$, the restricted map
\[
\nu \in M^+_1(X) \longmapsto \wkslim_n A^*_{F_n} \nu,
\]
is well-defined and weak*-to-weak* continuous.
\end{remark}

\section{Quotients and factors}
\label{sec:quotients-factors}

In this section, we discuss quotients and factors. We begin by showing that probability measures detect continuity, while subprobability measures detect continuity together with properness. We then apply this principle to the compactified ergodic map $\Phi_\infty$ and, finally, to proper factor maps.

\subsection{Proper maps and related properties}

Let $X$ and $Y$ be separable locally compact metric spaces, and let $F\colon X\to Y$ be a map. Whenever $f\circ F\in C_b(X)$ for every $f\in C_b(Y)$, composition with $F$ defines the \textbf{pullback operator}
\[
F^*\colon C_b(Y)\longrightarrow C_b(X), \qquad f \longmapsto F^*f:=f\circ F.
\]
If $F^*$ is well-defined, then it is automatically linear and contractive:
\[
\|F^*f\|_\infty\leq\|f\|_\infty \qquad \text{for every }f\in C_b(Y).
\]
If $F$ is Borel measurable, the \textbf{pushforward} of $\nu\in M(X)$ under $F$ is the measure $F_*\nu\in M(Y)$ defined by
\[
(F_*\nu)(B):=\nu(F^{-1}(B)) \qquad \text{for every Borel set }B\subseteq Y.
\]
This defines a linear operator $F_*\colon M(X)\to M(Y)$, $\nu\mapsto F_*\nu$, whose restrictions to $M_{\leq 1}^+(X)$ and $M_1^+(X)$ are affine maps into $M_{\leq 1}^+(Y)$ and $M_1^+(Y)$, respectively. The \textbf{fiber correspondence} associated with $F$ is the set-valued map
\[
F^{-1}\colon Y\to 2^X, \qquad y \longmapsto F^{-1}(y) := F^{-1}(\{y\}) = \{x \in X: F(x) = y\},
\]
where empty fibers are allowed. We say that $F^{-1}$ has \textbf{closed values} if $F^{-1}(y)$ is a closed subset of $X$ for each $y \in Y$, and that it is \textbf{upper hemicontinuous} if, whenever $F^{-1}(y)\subseteq U$ with $U\subseteq X$ open, there exists a neighborhood $V$ of $y$ such that $F^{-1}(z)\subseteq U$ for every $z\in V$.

\begin{lemma}
\label{lem:compact-continuity-characterizations}
Let $X$ and $Y$ be compact metric spaces, and let $F\colon X\to Y$ be a map. The following are equivalent:
\begin{enumerate}
\item $F$ is continuous.
\item $F^*\colon C(Y)\to C(X)$ is well-defined.
\item $F_*\colon M_1^+(X)\to M_1^+(Y)$ is well-defined and weak*-to-weak* continuous.
\item $F^{-1} \colon Y \to 2^X$ has closed values and is upper hemicontinuous.
\end{enumerate}
\end{lemma}

\begin{proof}
We establish the equivalences $(1) \iff (2)$, $(1) \iff (3)$, and $(1) \iff (4)$.

\medskip
\noindent
\textbf{$(1) \iff (2)$.} Since $X$ and $Y$ are compact, we have $C_b(X)=C(X)$ and $C_b(Y)=C(Y)$. As $C(Y)$ generates the topology of $Y$, the equivalence follows from \cite[Proposition~1.4.9]{engelking1989}.

\medskip
\noindent
\textbf{$(1) \iff (3)$.} If $F$ is continuous, then it is Borel measurable, so $F_*$ is well-defined. Its weak*-to-weak* continuity follows from \cite[Theorem~15.14]{aliprantis2006}. Conversely, assume that $F_*$ is well-defined and weak*-to-weak* continuous. Let $\iota_X, \iota_Y$ be the corresponding Dirac embeddings. Since $F_*\circ\iota_X=\iota_Y\circ F$, we have $F = \iota_Y^{-1}\circ F_*\circ\iota_X$, where $\iota_Y^{-1}$ is defined on the subspace $\iota_Y(Y)\subseteq M_1^+(Y)$. Since $\iota_X$ and $F_*$ are continuous and $\iota_Y^{-1}$ is continuous on $\iota_Y(Y)$, it follows that $F$ is continuous.

\medskip
\noindent
\textbf{$(1) \iff (4)$.} Let $\operatorname{Gr}(F) = \{(x,y)\in X\times Y:y=F(x)\}$ and $\operatorname{Gr}(F^{-1}) = \{(y,x)\in Y\times X:x\in F^{-1}(y)\}$. The homeomorphism $\tau\colon X\times Y\to Y\times X$, $\tau(x,y):=(y,x)$, satisfies $\tau(\operatorname{Gr}(F)) = \operatorname{Gr}(F^{-1})$. By the Closed Graph Theorem~\cite[Theorem~2.58]{aliprantis2006}, $F$ is continuous if and only if $\operatorname{Gr}(F)$ is closed. On the other hand, since $X$ is compact Hausdorff, the Closed Graph Theorem for correspondences \cite[Theorem~17.11]{aliprantis2006} gives $\operatorname{Gr}(F^{-1})$ is closed if and only if $F^{-1}$ has closed values and is upper hemicontinuous. This proves the equivalence.
\end{proof}

We next give equivalent characterizations of continuity and properness.

\begin{lemma}
\label{lem:continuity-proper-characterizations}
Let $X$ and $Y$ be separable locally compact metric spaces, and let
$F\colon X\to Y$ be a map. The following are equivalent:
\begin{enumerate}
\item $F$ is continuous and proper.
\item $F$ admits a continuous extension $\widehat{F}\colon \widehat{X}\to\widehat{Y}$ such that $\widehat{F}(\infty_X)=\infty_Y$.
\item $F^*\colon C_0(Y)\to C_0(X)$ is well-defined.
\item $F_*\colon M_{\leq 1}^+(X)\to M_{\leq 1}^+(Y)$ is well-defined, affine, and weak*-to-weak* continuous.
\end{enumerate}
\end{lemma}

\begin{proof}
Let $\widehat{F}\colon\widehat{X}\to \widehat{Y}$ denote the set-theoretic extension of $F$ defined by $\widehat{F}|_X=F$ and $\widehat F(\infty_X)=\infty_Y$. We establish the equivalences $(1) \iff (2)$, $(2) \iff (3)$, and $(2) \iff (4)$.

\medskip
\noindent
\textbf{$(1) \iff (2)$.} See \cite[Ch. I \S~10.3]{bourbaki1989}.

\medskip
\noindent
\textbf{$(2) \iff (3)$.} If
$f\in C_0(Y)$, then $\overline{f\circ F} = \overline f\circ\widehat F$, where $\overline{f\circ F}$ and $\overline{f}$ denote the extensions by zero in $\widehat{X}$ and $\widehat{Y}$, respectively. Hence continuity of $\widehat{F}$ implies $f\circ F\in C_0(X)$. Conversely, assume that $F^*\colon C_0(Y)\to C_0(X)$ is well-defined. Every $g\in C(\widehat{Y})$ can be written uniquely as $g=\overline f+g(\infty_Y)\mathbbm{1}$ for some $f\in C_0(Y)$. Therefore,
\[
g\circ\widehat F = \overline{F^*f}+g(\infty_Y)\mathbbm{1} \in C(\widehat X).
\]
Thus pullback by $\widehat F$ maps $C(\widehat Y)$ into $C(\widehat X)$. By Lemma~\ref{lem:compact-continuity-characterizations}, the map $\widehat F$ is continuous.

\medskip
\noindent
\textbf{$(2) \iff (4)$.} For $Z=X,Y$, let $\iota_Z\colon M_{\leq 1}^+(Z)\to M_1^+(\widehat Z)$ be the affine weak*-to-weak* homeomorphism
\[
\iota_Z(\eta) := (j_Z)_*\eta+\left(1-\eta(Z)\right)\delta_{\infty_Z}.
\]
The map $F$ is Borel measurable if and only if $\widehat F$ is Borel measurable. Whenever the pushforwards are defined, as $\widehat{F} \circ j_X = j_Y \circ F$, for any $\eta \in M^+_{\leq 1}(X)$,
\begin{align*}
\widehat{F}_*(\iota_X(\eta)) & = \widehat{F}_*((j_X)_*\eta+
\left((1-\eta(X)\right))\delta_{\infty_X}) \\
    &   = (j_Y)_* (F_*\eta)+
\left((1-(F_*\eta)(Y)\right))\delta_{\infty_Y} = \iota_Y(F_*\eta).
\end{align*}
Thus $\widehat{F}_* \circ \iota_X = \iota_Y \circ F_*$. Since $\iota_X$ and $\iota_Y$ are affine weak*-to-weak* homeomorphisms, condition $(4)$ is equivalent to $\widehat F_*\colon M_1^+(\widehat X)\to M_1^+(\widehat Y)$ being well-defined and weak*-to-weak* continuous. By Lemma~\ref{lem:compact-continuity-characterizations}, this is equivalent to the continuity of $\widehat F$.
\end{proof}

Denote by $X/F$ the quotient space associated with the equivalence relation
\[
x\sim_F x' \qquad \iff \qquad F(x)=F(x').
\]
Endow $X/F$ with the quotient topology, and let $\pi_F\colon X\to X/F$, $\pi_F(x):=[x]_F$, be the canonical quotient map. Endow $F(X)$ with the subspace topology inherited from $Y$, and define
\[
\widetilde{F}\colon X/F\longrightarrow F(X),
\qquad
\widetilde{F}([x]_F):=F(x).
\]
Then $\widetilde{F}$ is a well-defined bijection and $F=\widetilde F\circ\pi_F$.

\begin{lemma}
\label{lem:compact-quotient-map}
Let $X$ and $Y$ be compact metric spaces, and let $F\colon X\to Y$ be continuous. Then $\widetilde{F}\colon X/F \to F(X)$ is a homeomorphism. Hence, the pushforward $\widetilde{F}_*\colon M(X/F) \to M(F(X))$ is a linear weak*-to-weak* homeomorphism, and its restrictions to $M_{\leq 1}^+$ and $M_1^+$ are affine weak*-to-weak* homeomorphisms.
\end{lemma}

\begin{proof}
Since $F=\widetilde{F}\circ\pi_F$ and $\pi_F$ is a quotient map, the universal property of the quotient topology implies that $\widetilde F$ is continuous (see \cite[Theorem~22.2]{munkres2000}). Moreover, $X/F$ is compact, being the continuous image of $X$ under $\pi_F$, while $F(X)$ is Hausdorff as a subspace of $Y$. Hence the continuous bijection $\widetilde{F}\colon X/F\to F(X)$ is a homeomorphism by \cite[Theorem~26.6]{munkres2000}. Thus, the pullback $\widetilde{F}^*\colon C(F(X)) \to C(X/F)$, $g\mapsto g\circ\widetilde F$, is an isometric isomorphism whose inverse is $(\widetilde F^{-1})^*$. By the Change of Variables Theorem~\cite[Theorem~13.46]{aliprantis2006}, first for positive measures and then by linearity and the Jordan decomposition, $\widetilde F_*=(\widetilde F^*)^*$. Similarly, $(\widetilde F^{-1})_* = \left(((\widetilde F^{-1})^*\right))^*$. Both maps are weak*-to-weak* continuous by the general weak*continuity of adjoints recalled in Section~\ref{sec:linear-operator}, and they are inverse to each other. Therefore, $\widetilde{F}_*\colon M(X/F)\to M(F(X))$ is a linear weak*-to-weak* homeomorphism. Since both pushforwards preserve positivity and total mass, $\widetilde{F}_*$ restricts to affine weak*-to-weak* homeomorphisms on $M_{\leq 1}^+$ and $M_1^+$.
\end{proof}

\subsection{The compactified ergodic quotient}

We now describe the space of ergodic measures as a quotient of the compactified system. We also identify the range of the induced pullback with the invariant functions and the kernel of the induced pushforward with the weak* coboundaries.

\begin{proposition}
\label{prop:compactified-ergodic-quotient}
Let $S$ be a semigroup, let $X$ be a separable locally compact metric space, and let $S\acts X$ be an action by proper maps. Let $S\acts\widehat X$ be the compactified action fixing $\infty$. Assume that $S\acts X$ is pointwise uniquely ergodic. Let $\Phi_\infty\colon \widehat X\to M_1^+(\widehat X,S)$ be the ergodic map of the compactified system. The following are equivalent:
\begin{enumerate}
\item $S\acts X$ is continuously pointwise ergodic and vanishes at infinity.
\item $\Phi_\infty\colon \widehat X\to M_1^+(\widehat X,S)$ is weak* continuous.
\item $(\Phi_\infty)_*\colon M_1^+(\widehat X)\to M_1^+\left((M_1^+(\widehat X,S)\right))$ is well-defined and weak*-to-weak* continuous.
\item $\Phi_\infty^{-1}\colon M_1^+(\widehat X,S)\to 2^{\widehat X}$
has closed values and is upper hemicontinuous.
\end{enumerate}
\end{proposition}

\begin{proof}
The equivalence between $(1)$ and $(2)$ follows from the criterion established in \S~\ref{subsection:pointwise-ergodicity}: $\Phi_\infty$ is weak* continuous if and only if $\Phi$ is weak* continuous and vanishes at infinity. Since the compactified action is pointwise uniquely ergodic, weak* continuity of $\Phi_\infty$ is exactly continuous pointwise ergodicity of $S\acts\widehat X$. The equivalence between $(2)$, $(3)$, and $(4)$ follows from Lemma~\ref{lem:compact-continuity-characterizations} applied to $\Phi_\infty$.
\end{proof}

If $X$ is compact and $M^+_1(X,S)$ is non-empty, the Choquet simplex $M^+_1(X,S)$ is said to be a \textbf{Bauer simplex} if $E^+_1(X,S)$ is closed in $M^+_1(X,S)$. In this case, the \textbf{resultant map}
\[
r\colon M^+_1(E^+_1(X,S)) \longrightarrow M^+_1(X,S), \qquad r(\lambda) := \int_{E^+_1(X,S)} {\operatorname{id}_{M(X)}}\, d\lambda,
\]
is an affine homeomorphism \cite[Proposition~11.1]{phelps2001}. Its inverse $r^{-1}$ assigns, to each $\mu \in M^+_1(X,S)$, the measure $\lambda = r^{-1}(\mu)$, called the \textbf{ergodic decomposition} of $\mu$. 

\begin{theorem}
\label{thm:main-quotient}
Let $S$ be a semigroup, let $X$ be a separable locally compact metric space, and let $S\acts X$ be an action by proper maps. Assume that $S\acts X$ is continuously pointwise ergodic and vanishes at infinity. Let $\Phi_\infty\colon \widehat{X}\to M_1^+(\widehat{X},S)$ be the ergodic map of the compactified system, let $\pi_\infty \colon \widehat{X}\to \widehat{X}/\Phi_\infty$ be the quotient map, and let $\widetilde{\Phi}_\infty\colon \widehat{X}/\Phi_\infty\to \Phi_\infty(\widehat{X})$ be the induced bijection. Then the following assertions hold.
\begin{enumerate}
\item $E_1^+(\widehat{X},S)$ is homeomorphic to the quotient space $\widehat{X}/\Phi_\infty$. Moreover, $M_1^+(\widehat{X},S)$ is a Bauer simplex affinely weak* homeomorphic to $M_1^+(\widehat{X}/\Phi_\infty)$.
\item For every $\nu \in M_1^+(\widehat{X})$,
\[
P_\infty^*\nu = \int_{\widehat{X}}\Phi_\infty\,d\nu = \int_{E_1^+(\widehat{X},S)} \operatorname{id}_{M(\widehat{X})}\,d(\Phi_\infty)_*\nu,
\]
where $P_\infty$ is the mean ergodic projection of the compactified system. In particular, if $\mu \in M_1^+(\widehat{X},S)$, then $(\Phi_\infty)_*\mu$ is its ergodic decomposition, and
\[
r_\infty^{-1}
=
\left.(\Phi_\infty)_*\right|_{M_1^+(\widehat{X},S)},
\qquad
r_\infty\circ(\Phi_\infty)_* = \left.P_\infty^*\right|_{M_1^+(\widehat{X})},
\]
where $r_\infty\colon
M_1^+(E_1^+(\widehat{X},S))\to M_1^+(\widehat{X},S)$ denotes the resultant map.

\item The pullback $\pi_\infty^*\colon C(\widehat{X}/\Phi_\infty)\to C(\widehat{X})$ is an isometric isomorphism onto $\Fix(\kappa_\infty)$. The pushforward $(\pi_\infty)_*\colon
M(\widehat X)\to M(\widehat X/\Phi_\infty)$ is a linear weak*-to-weak* continuous surjection and $\ker((\pi_\infty)_*) = \wsCob(\kappa_\infty^*)$. Moreover, $M(\widehat{X},S) \cong M(\widehat X/\Phi_\infty)$.
\end{enumerate}
\end{theorem}

\begin{proof}
We prove the three statements separately.

\medskip
\noindent
\emph{Statement (1).} Since $\Phi_\infty\colon \widehat{X} \to M_1^+(\widehat{X},S)$ is continuous and $\widehat{X}$ is compact, $\Phi_\infty(\widehat{X})$ is compact as well. By Proposition~\ref{prop:transitive-point}, $\Phi_\infty(\widehat{X}) = E_1^+(\widehat{X},S)$, and Lemma~\ref{lem:compact-quotient-map} implies that $\widehat{X}/\Phi_\infty \cong \Phi_\infty(\widehat{X}) = E_1^+(\widehat{X},S)$. Hence $\widehat{X}/\Phi_\infty$ is compact and metrizable. Moreover, as $E_1^+(\widehat{X},S)$ is compact (hence closed in $M_1^+(\widehat{X},S)$), $M_1^+(\widehat{X},S)$ is a Bauer simplex. By Lemma~\ref{lem:compact-quotient-map}, $\widetilde{\Phi}_\infty$ is a homeomorphism that induces an affine map $(\widetilde{\Phi}_\infty)_*\colon M_1^+(\widehat{X}/\Phi_\infty)\to M_1^+(E_1^+(\widehat{X},S))$ that is a weak*-to-weak* homeomorphism. Since $M_1^+(\widehat{X},S)$ is Bauer, the resultant map $r_\infty$ is an affine weak*-to-weak* homeomorphism. This proves the first two assertions.

\medskip
\noindent
\emph{Statement (2).} By Theorem~\ref{thm:main1} applied to $S\acts\widehat{X}$, $P_\infty^*\nu = \int_{\widehat X}\Phi_\infty\,d\nu$. Since $\Phi_\infty(\widehat X)=E_1^+(\widehat{X},S)$, for every $f\in C(\widehat X)$ the Change of Variables Theorem~\cite[Theorem~13.46]{aliprantis2006} gives
\begin{align*}
\left\langle
f,\int_{\widehat X}\Phi_\infty\,d\nu \right\rangle
&=
\int_{\widehat X}
\left\langle f,\Phi_\infty(x)\right\rangle\,d\nu(x) \\
&=
\int_{E_1^+(\widehat X,S)}
\left\langle f,\mu\right\rangle\,
d((\Phi_\infty)_*\nu)(\mu) =
\left\langle
f,
\int_{E_1^+(\widehat X,S)}
\operatorname{id}_{M(\widehat X)}
\,d(\Phi_\infty)_*\nu
\right\rangle.
\end{align*}
Since $f$ is arbitrary, the two Gelfand integrals coincide. If $\mu\in M_1^+(\widehat{X},S)$, then $P_\infty^*\mu=\mu$. Hence
\[
\mu = P_\infty^*\mu = \int_{\widehat X}\Phi_\infty\,d\mu = \int_{E_1^+(\widehat{X},S)} \operatorname{id}_{M(\widehat{X})}\,d(\Phi_\infty)_*\mu.
\]
By uniqueness of the ergodic decomposition, $(\Phi_\infty)_*\mu=r_\infty^{-1}(\mu)$. This gives the identities involving $r_\infty$.

\medskip
\noindent
\emph{Statement (3).} If $\widetilde{f}\in C(\widehat{X}/\Phi_\infty)$, then $\widetilde{f} \circ \pi_\infty$ is $\kappa_\infty$-invariant because $\pi_\infty(sx)=\pi_\infty(x)$ for all $s\in S$, $x\in\widehat{X}$. Thus $\ran(\pi_\infty^*) \subseteq \Fix(\kappa_\infty)$. Conversely, let $f \in \Fix(\kappa_\infty)$. Then $f$ is constant on each orbit closure. If $\pi_\infty(x)=\pi_\infty(x')$, then $\Phi_\infty(x) = \Phi_\infty(x')$. The common measure has non-empty support and
\[
\operatorname{supp}(\Phi_\infty(x)) \subseteq \operatorname{cl}_{\widehat{X}}(Sx) \cap \operatorname{cl}_{\widehat{X}}(Sx').
\]
Hence the constants taken by $f$ on the two orbit closures agree, and $f(x)=f(x')$. Therefore $f$ descends to a unique function on $\widehat{X}/\Phi_\infty$. Since $\pi_\infty$ is a quotient map, this function is continuous. Thus $\ran(\pi_\infty^*)=\Fix(\kappa_\infty)$. Since $\pi_\infty$ is surjective, $\ker(\pi_\infty^*)=\{0\}$. Finally, $\pi_\infty$ is a continuous surjection between compact Hausdorff spaces, so $\pi_\infty^*$ is an isometric embedding. Hence its adjoint $(\pi_\infty)_*=(\pi_\infty^*)^*$ is a linear weak*-to-weak* continuous surjection. Moreover,
\[
\ker((\pi_\infty)_*) = \annil{\ran(\pi_\infty^*)} = \annil{\Fix(\kappa_\infty)} = \wsCob(\kappa_\infty^*),
\]
where the last equality follows from Proposition~\ref{prop:ann-cob-koopman}. Since the compactified system is continuously pointwise ergodic, Theorem~\ref{thm:main1} gives
\[
M(\widehat X) = M(\widehat X,S)\oplus\wsCob(\kappa_\infty^*).
\]
Therefore, the restriction $\left.(\pi_\infty)_*\right|_{M(\widehat X,S)}\colon M(\widehat X,S) \to M(\widehat X/\Phi_\infty)$ is bijective. To see that the inverse is weak*-to-weak* continuous, set
\[
T_\infty := (\pi_\infty^*)^{-1}\circ P_\infty \colon C(\widehat X)\longrightarrow C(\widehat X/\Phi_\infty),
\]
where $(\pi_\infty^*)^{-1}$ is defined on
$\ran(\pi_\infty^*)=\Fix(\kappa_\infty)$. Then $T_\infty$ is bounded, so $T_\infty^*$ is weak*-to-weak* continuous. Moreover, $T_\infty\pi_\infty^*=I$ and $\pi_\infty^*T_\infty=P_\infty$. Taking adjoints and using $P_\infty^*\mu=\mu$ for $\mu\in M(\widehat X,S)$ shows that $T_\infty^*
=
\left(\left.(\pi_\infty)_*\right|_{M(\widehat X,S)}\right)^{-1}$. Thus the restriction of $(\pi_\infty)_*$ is a linear weak*-to-weak* homeomorphism. Hence, $M(\widehat X,S) \cong M(\widehat X/\Phi_\infty)$.
\end{proof}

\begin{remark}
If $X$ is compact, the vanishing-at-infinity condition is automatic, the point $\infty$ is isolated in $\widehat{X} = X \sqcup \{\infty\}$, and Theorem~\ref{thm:main-quotient} also holds with $\widehat{X}$ replaced by $X$ and all the $\infty$ subscripts omitted.
\end{remark}

\subsection{Proper factor maps}

Let $S\acts X$ and $S\acts Y$ be actions by proper maps on locally compact metric spaces. A map $\pi\colon X\to Y$ is called \textbf{$S$-equivariant} if
\[
\pi(sx)=s\pi(x) \qquad \text{for all }s\in S,~x\in X.
\]
A \textbf{factor map} is a continuous, surjective, $S$-equivariant map. In this case we say that $S\acts Y$ is a \textbf{topological factor} of $S\acts X$. A factor map $\pi\colon X\to Y$ is called a \textbf{proper factor map} if $\pi$ is proper. When $X$ is compact, every factor map $\pi\colon X\to Y$ is automatically proper.

\begin{proposition}
\label{prop:pue-descends-proper-factors}
Let $S$ be a countable, discrete, bicancellative, and left amenable semigroup. Let $S\acts X$ and $S\acts Y$ be actions by proper maps on separable locally compact metric spaces, and let $\pi\colon X\to Y$ be a proper factor map.

If $S\acts X$ is pointwise uniquely ergodic, then $S\acts Y$ is pointwise uniquely ergodic. Moreover, their pointwise ergodic maps satisfy
\[
\Phi_Y\circ\pi=\pi_*\circ\Phi_X.
\]
If, in addition, $S\acts X$ is continuously pointwise ergodic, then $S\acts Y$ is continuously pointwise ergodic. If $S\acts X$ also vanishes at infinity, then so does $S\acts Y$. In consequence, if $S\acts X$ is uniform, then $S\acts Y$ is uniform.
\end{proposition}

\begin{proof}
By Lemma~\ref{lem:continuity-proper-characterizations}, $\pi$ admits a continuous extension $\widehat{\pi}\colon\widehat{X} \to \widehat{Y}$ such that $\widehat{\pi}|_X=\pi$ and $\widehat{\pi}(\infty_X)=\infty_Y$. This extension is surjective and $S$-equivariant.

Assume first that $S\acts X$ is pointwise uniquely ergodic. Then the compactified action $S\acts\widehat X$ is pointwise uniquely ergodic. Let $\Phi_{\widehat{X}}$ denote its ergodic map. Fix $y\in\widehat Y$ and choose $x\in\widehat X$ such that $\widehat{\pi}(x)=y$. Set $K_x:=\operatorname{cl}_{\widehat X}(Sx)$, $L_y:=\operatorname{cl}_{\widehat Y}(Sy)$. By equivariance of $\widehat{\pi}$, $\widehat{\pi}(Sx)=S\widehat{\pi}(x)=Sy$. Continuity of $\widehat{\pi}$ and compactness of $K_x$ give $\widehat{\pi}(K_x)=L_y$. Hence, $\widehat{\pi}_*\colon M_1^+(K_x)\to M_1^+(L_y)$ is surjective \cite[Theorem~15.14]{aliprantis2006}. Thus, every probability measure on $L_y$ admits a probability lift to $K_x$.

The measure $\widehat{\pi}_*\Phi_{\widehat X}(x)$ belongs to $M_1^+(L_y,S)$, so this set is non-empty. We show that it is its only element. Let $\mu_y \in M_1^+(L_y,S)$ and choose $\nu_x \in M_1^+(K_x)$ such that $\widehat{\pi}_*\nu_x = \mu_y$. Fix a left F{\o}lner sequence $(F_n)_n$. Any weak* accumulation point $\mu_x$ of $(A_{F_n}^*\nu_x)_n$ is $S$-invariant as in Lemma~\ref{lem:krylov}. Moreover, equivariance and the invariance of $\mu_y$ give
\[
\widehat{\pi}_*A_{F_n}^{*}\nu_x = A_{F_n}^*\mu_y = \mu_y,
\]
and hence, by continuity of $\widehat{\pi}_*$ (see Lemma~\ref{lem:compact-continuity-characterizations}), $\widehat{\pi}_*\mu_x=\mu_y$. Pointwise unique ergodicity of $S\acts\widehat{X}$ implies $\mu_x = \Phi_{\widehat{X}}(x)$. In consequence, $\mu_y=\widehat{\pi}_*\Phi_{\widehat{X}}(x)$. Thus $L_y$ supports a unique invariant probability measure. We conclude that $S\acts\widehat{Y}$ is pointwise uniquely ergodic and that its ergodic map satisfies
\[
\Phi_{\widehat{Y}}\circ\widehat\pi = \widehat\pi_*\circ\Phi_{\widehat{X}}.
\]
By the equivalence between pointwise unique ergodicity of an action and of its compactification (see \S~\ref{subsection:pointwise-ergodicity}), $S\acts Y$ is pointwise uniquely ergodic. Under the canonical identifications of subprobabilities with probabilities on the compactifications, the preceding identity restricts to
\[
\Phi_Y\circ\pi=\pi_*\circ\Phi_X.
\]

Assume that $\Phi_X$ is weak* continuous. Since $\pi$ is continuous and proper, by Lemma~\ref{lem:continuity-proper-characterizations}, $\pi_*\colon M_{\leq 1}^+(X)\to M_{\leq 1}^+(Y)$ is weak*-to-weak* continuous. Hence $\pi_*\circ\Phi_X$ is weak* continuous. Moreover, by \cite[Theorem~A.38 and Theorem~A.57]{lee2013}, a continuous proper surjection between locally compact Hausdorff spaces is a quotient map. The identity $\Phi_Y\circ\pi=\pi_*\circ\Phi_X$ therefore implies that $\Phi_Y$ is weak* continuous (see \cite[Theorem~22.2]{munkres2000}). Next, suppose that $S\acts X$ also vanishes at infinity. Then $\Phi_{\widehat{X}}$ is weak* continuous by the compactification criterion. The identity $\Phi_{\widehat{Y}}\circ\widehat\pi = \widehat\pi_*\circ\Phi_{\widehat{X}}$ has a weak* continuous right-hand side. Since $\widehat\pi$ is a continuous surjection between compact Hausdorff spaces, it is a quotient map, and hence $\Phi_{\widehat{Y}}$ is weak* continuous. Applying the compactification criterion once more shows that $S\acts Y$ vanishes at infinity.

Finally, assume that $S\acts X$ is uniform. By Theorem~\ref{thm:main2}, uniformity of $S\acts X$ is equivalent to continuous pointwise ergodicity together with vanishing at infinity. By the previous results, these properties descend to $S\acts Y$. Another application of Theorem~\ref{thm:main2} shows that $S\acts Y$ is uniform.
\end{proof}

\begin{remark}
\label{remark:proper_factors}
Neither properness nor left amenability can be omitted from
Proposition~\ref{prop:pue-descends-proper-factors} in general.

To see that properness is essential, let $S$ be an infinite countable bicancellative semigroup, and let $S\acts Y$ be any action by proper maps. Endow $S$ with the discrete topology, set $X:=Y\times S$, and consider the diagonal action $s\cdot(y,t):=(sy,st)$. This is again an action by proper maps. By bicancellativity, every orbit in $X$ is a closed infinite discrete subset and carries no $S$-invariant probability measure. Indeed, any probability measure on such a countable orbit has an atom of positive mass, while invariance and right cancellativity would produce infinitely many atoms of the same positive mass. Hence, $S\acts X$ is continuously pointwise ergodic, with ergodic map $\Phi_X\equiv0$, and vanishes at infinity. The projection $\pi\colon X\to Y$, $\pi(y,t):=y$, is a factor map, but it is not proper, since $\pi^{-1}(\{y\})=\{y\}\times S$ is non-compact. Thus any action of $S$ by proper maps can occur as a non-proper factor of a continuously pointwise ergodic action that vanishes at infinity. In particular, taking $S=\Z$ and $Y=\{0,1\}^{\Z}$ with the shift shows that properness cannot be omitted even when $S$ is left amenable.

Left amenability is also essential, even in the compact setting. Indeed, Glasner and Weiss construct, for
$G=\operatorname{SL}(2,\Z)$, compact metric minimal actions $G\acts X$ and $G\acts Y$, together with a factor map $\pi\colon X\to Y$, such that $G\acts X$ is strictly ergodic, while $G\acts Y$ is not uniquely ergodic (see \cite[Theorem~3.3]{glasner2017}). Since both systems are minimal, $G\acts X$ is continuously pointwise ergodic, while $G\acts Y$ is not pointwise uniquely ergodic. Here $\pi$ is automatically proper, since $X$ is compact, but the acting group is
not amenable.
\end{remark}

\section{Examples and applications}
\label{sec:examples-applications}

We conclude with examples and applications of the preceding theory. An interval map first separates pointwise unique ergodicity from continuous pointwise ergodicity and motivates a three-step surgery by invariant deletions. Orientation-preserving circle homeomorphisms then illustrate the dependence on the acting semigroup, while rigid rotations allow explicit computation of the ergodic quotient, the mean ergodic projection, and the invariant and coboundary spaces. Translations model escape to infinity, and a finite action shows that the pointwise invariant-measure admitting hypothesis cannot be omitted. We finish with applications to entropy and topological emergence.

\subsection{A well-defined but discontinuous ergodic map}
\label{subsec:interval-example}

Let $T\colon[0,1]\to[0,1]$ be continuous and assume that
\[
T(0)=0, \qquad T(1)=1, \qquad 0<T(x)<x \quad \text{for every }x\in(0,1).
\]
We consider the action of $\N$ generated by $T$. For every $x\in[0,1)$, the sequence $(T^nx)_n$ decreases to $0$. Hence
\[
\Phi(x)
=
\begin{cases}
\delta_0,&x\in[0,1),\\
\delta_1,&x=1.
\end{cases}
\]
The action is pointwise uniquely ergodic, but it is not continuously pointwise ergodic. The invariant probability measures on $[0,1]$ are precisely
\[
M_1^+([0,1],\N) = \left\{t\delta_0+(1-t)\delta_1:t\in[0,1]\right\}.
\]

The only non-empty fibers of $\Phi^{-1}\colon M_1^+([0,1],\N)\to 2^{[0,1]}$ are $\Phi^{-1}(\delta_0)=[0,1)$ and $\Phi^{-1}(\delta_1)=\{1\}$. The fiber map is upper hemicontinuous but not closed-valued, since $\Phi^{-1}(\delta_0)$ is not closed. Compare Lemma~\ref{lem:compact-continuity-characterizations}.

The quotient $[0,1]/\Phi$ consists of the two classes $[0,1)$ and $\{1\}$, and its open sets are $\varnothing$, $\{[0,1)\}$, $[0,1]/\Phi$. Thus $[0,1]/\Phi$ is the Sierpi\'nski space and, in particular, is not Hausdorff. This illustrates why continuity of the ergodic map is needed to identify the quotient with a subspace of the space of measures. See Theorem~\ref{thm:main-quotient}.

If $T$ is a homeomorphism, one can consider the bilateral action $\Z\acts[0,1]$, $n\cdot x:=T^nx$. For $x\in(0,1)$, the inequality $T(x)<x$ implies $T^nx\to 0$ and $T^{-n}x\to 1$ as $n\to+\infty$. In consequence, $\operatorname{cl}_{[0,1]}(\Z x) = \{T^nx:n\in\Z\}\cup\{0,1\}$. The invariant probability measures on this orbit closure are precisely
\[
M_1^+\left((\operatorname{cl}_{[0,1]}(\Z x),\Z\right))
=
\left\{
t\delta_0+(1-t)\delta_1:
t\in[0,1]
\right\}.
\]
Therefore, unlike the forward $\N$-action, the bilateral $\Z$-action is not pointwise uniquely ergodic. This failure can also be seen from the dependence of orbit averages on the chosen F{\o}lner sequence. For every $x\in(0,1)$, as $n \to \infty$,
\[
\frac{1}{n}\sum_{k=0}^{n-1}\delta_{T^kx}
\longrightarrow\delta_0 \qquad \text{and} \qquad
\frac1n\sum_{k=-n+1}^{0}\delta_{T^kx}
\longrightarrow\delta_1.
\]
The two displayed F{\o}lner sequences yield different limits. Alternating them produces a F{\o}lner sequence along which the orbit averages do not converge. Hence the action is neither uniform nor weak* mean continuous, in agreement with Theorem~\ref{thm:main2}.

\subsection{Surgery by invariant deletions}
\label{subsec:surgeries}

The preceding example suggests a three-step procedure for extracting an increasingly regular ergodic core: first remove the indeterminacy of the ergodic map, then its discontinuities, and finally the invariant mass that persists at infinity. The resulting cores may be empty, and no maximality is claimed.

For an action $S\acts Y$ and $D\subseteq Y$, let $\operatorname{int}_{\mathrm{ci}}^Y(D)$ denote the union of all open sets $V\subseteq D$ such that $s^{-1}V=V$ for every $s\in S$. Thus $\operatorname{int}_{\mathrm{ci}}^Y(D)$ is the largest open completely $S$-invariant subset of $Y$ contained in $D$. Restrictions of actions by proper maps to such sets remain actions by proper maps.

Let now $S\acts X$ be a pointwise invariant-measure admitting action by continuous maps on a compact metric space. Thus every orbit closure supports an invariant probability measure. This assumption is automatic under the amenability hypotheses of Lemma~\ref{lem:krylov}.

\medskip
\noindent
\textbf{Step 1: removing indeterminacy.}
Set
\[
U_1 := \operatorname{int}_{\mathrm{ci}}^X
\left(\left\{x\in X: M_1^+\left((\operatorname{cl}_X(Sx),S\right))
\text{ is a singleton}\right\}\right).
\]
For $x\in U_1$, let $\mu_x$ be the unique invariant probability measure on $\operatorname{cl}_X(Sx)$ and define
\[
\Phi_{U_1}\colon U_1\longrightarrow M_{\leq 1}^+(U_1,S), \qquad \Phi_{U_1}(x):=\left.\mu_x\right|_{U_1}.
\]
Then $S\acts U_1$ is pointwise uniquely ergodic, with ergodic map $\Phi_{U_1}$. Indeed, compact orbit closures contained in $U_1$ retain their unique invariant probability measure, while invariant mass carried by the deleted boundary becomes the zero subprobability after restriction.

\medskip
\noindent
\textbf{Step 2: removing discontinuities.}
Set
\[
U_2
:=
\operatorname{int}_{\mathrm{ci}}^{U_1}
\left(
\left\{
x\in U_1:
\Phi_{U_1}\text{ is weak* continuous at }x
\right\}
\right).
\]
The restricted action $S\acts U_2$ is continuously pointwise ergodic, with ergodic map
\[
\Phi_{U_2}\colon U_2\longrightarrow M_{\leq 1}^+(U_2,S),
\qquad
\Phi_{U_2}(x):=\left.\Phi_{U_1}(x)\right|_{U_2}.
\]
Indeed, pointwise unique ergodicity is preserved under completely invariant deletion, and restriction of measures from $U_1$ to $U_2$ is weak*-to-weak* continuous.

\medskip
\noindent
\textbf{Step 3: removing persistent boundary mass.}
For a weak* continuous map
$\Psi\colon Y\to M_{\leq 1}^+(Y)$, let
\[
\operatorname{Acc}_\infty(\Psi)
:=
\left\{
\nu\in M_{\leq 1}^+(Y):
\begin{array}{c}
\text{there exists }y_n\to\infty\text{ in }Y\\
\text{such that }\Psi(y_n)\to\nu\text{ weak*}
\end{array}
\right\}.
\]
Set
\[
U_3
:=
\operatorname{int}_{\mathrm{ci}}^{U_2}
\left(
U_2\setminus
\bigcup_{\nu\in\operatorname{Acc}_\infty(\Phi_{U_2})}
\operatorname{supp}(\nu)
\right).
\]
Then $S\acts U_3$ is continuously pointwise ergodic, and its ergodic map
\[
\Phi_{U_3}\colon U_3\longrightarrow M_{\leq 1}^+(U_3,S),
\qquad
\Phi_{U_3}(x):=\left.\Phi_{U_2}(x)\right|_{U_3},
\]
vanishes at infinity. Indeed, a sequence escaping to infinity in $U_3$ either escapes to infinity already in $U_2$, in which case its limiting invariant mass was deleted, or accumulates on $U_2\setminus U_3$, in which case continuity and complete invariance force its limiting measure to be supported on the deleted boundary.

We have thus obtained $X\supseteq U_1\supseteq U_2\supseteq U_3$, where $S\acts U_1$ is pointwise uniquely ergodic, $S\acts U_2$ is continuously pointwise ergodic, and $S\acts U_3$ additionally vanishes at infinity.

Let $A:=X\setminus U_3$. Then $A$ is closed and completely $S$-invariant. If $A\neq\emptyset$, collapsing $A$ to one point gives a canonical identification $X/A\cong\widehat{U_3}$, under which the action induced on $X/A$ is the compactification of $S\acts U_3$. By Theorem~\ref{thm:locally-compact-main}, this compactified system is continuously pointwise ergodic.

Assume, in addition, that $S$ is countable, discrete, bicancellative, and left amenable, and that $A\neq\emptyset$. The functions on $X$ that descend to $X/A$ form the closed unital subalgebra
\[
\mathcal A_A
:=
\left\{
f\in C(X):
f|_A\text{ is constant}
\right\}
\cong C(X/A).
\]
For every left F{\o}lner sequence $(F_n)_n$ and every $f\in\mathcal A_A$, the averages
\[
A_{F_n}f(x)
=
\frac{1}{|F_n|}
\sum_{s\in F_n}f(sx)
\]
converge uniformly on $X$. More precisely, if $f|_A\equiv c$, then their uniform limit is
\[
L_Af(x)
=
\begin{cases}
\displaystyle
c+\int_{U_3}(f-c)\,d\Phi_{U_3}(x),
&x\in U_3,\\[2mm]
c,
&x\in A.
\end{cases}
\]

\medskip
\noindent
\textbf{The interval example.}
For the interval map above, the first step is vacuous, since every orbit closure is uniquely ergodic: $U_1=[0,1]$, $\Phi_{U_1}=\Phi$.

The only discontinuity of $\Phi_{U_1}$ occurs at $1$. Hence the second step gives $U_2=[0,1)$, $\Phi_{U_2}\equiv\delta_0$. Thus the restricted system on $[0,1)$ is continuously pointwise ergodic, but it does not vanish at infinity: if $x_n\to1$, then $\Phi_{U_2}(x_n)=\delta_0$. Hence, $\operatorname{Acc}_\infty(\Phi_{U_2}) = \{\delta_0\}$.

The third step therefore removes $\operatorname{supp}(\delta_0)=\{0\}$ and gives $U_3=(0,1)$, $\Phi_{U_3}\equiv 0$. Thus $[0,1]=U_1\supseteq U_2=[0,1)\supseteq U_3=(0,1)$
realizes the successive passage from pointwise unique ergodicity to continuous pointwise ergodicity, and then to continuous pointwise ergodicity with vanishing at infinity.

The total deleted set is $A=\{0,1\}$, and therefore $[0,1]/\{0,1\}\cong\mathbb{S}^1$, $\mathcal A_A
=
\{f\in C([0,1]):f(0)=f(1)\}$. It follows that, whenever $f(0)=f(1)=c$,
\[
\frac1n\sum_{j=0}^{n-1}f(T^jx)
\longrightarrow c
\]
uniformly on $[0,1]$. This condition is also necessary. Indeed, the pointwise limit of the averages is
\[
x\longmapsto
\begin{cases}
f(0),&x\in[0,1),\\
f(1),&x=1.
\end{cases}
\]
Uniform convergence would force this limit to be continuous, since every average is continuous. Hence
\[
\frac{1}{n}\sum_{j=0}^{n-1}f\circ T^j \text{ converges uniformly on }[0,1] \quad\iff\quad
f(0)=f(1).
\]

\subsection{Orientation-preserving circle homeomorphisms}
\label{subsec:circle-homeomorphisms}

Let $T\colon\mathbb{S}^1\to \mathbb{S}^1$, $\mathbb{S}^1:=\R/\Z$, be an orientation-preserving homeomorphism, and consider the forward action $\N\acts\mathbb{S}^1$ given by $n\cdot x:=T^n x$. Denote by $\rho(T)\in\R/\Z$ the rotation number of $T$. The following is a direct consequence of the classification of orientation-preserving circle homeomorphisms. See
\cite[Chapter~11]{katok1995}.

\begin{proposition}
\label{prop:circle-homeomorphisms}
The forward action $\N\acts\mathbb{S}^1$ is pointwise uniquely ergodic. Moreover:
\begin{enumerate}
\item Assume that $\rho(T)$ is irrational. Then $T$ is uniquely ergodic. Hence its ergodic map is constant and the action is continuously pointwise
ergodic.
\item Assume that $\rho(T)=\frac pq$, where $p$ and $q$ are coprime. Then every periodic orbit has period $q$, and the $\omega$-limit set of every $x\in\mathbb{S}^1$ is a periodic
orbit, denoted by $\omega(x)$. The ergodic map is
\[
\Phi(x) = m_{\omega(x)} :=
\frac{1}{q} \sum_{j=0}^{q-1}\delta_{T^jy},
\qquad \text{for}~y\in \omega(x).
\]
Furthermore:
\begin{enumerate}
\item $T$ is uniquely ergodic if and only if it has exactly one periodic orbit.
\item The forward action is continuously pointwise ergodic if and only if either $T$ has exactly one periodic orbit or $T^q=\operatorname{id}_{\mathbb{S}^1}$.
\item If $T$ has at least two periodic orbits and
$T^q\neq\operatorname{id}_{\mathbb{S}^1}$, then $\Phi$ is discontinuous.
\end{enumerate}
\end{enumerate}
\end{proposition}

\begin{proof}
The irrational case follows from \cite[Theorem~11.2.9]{katok1995}. Suppose that $\rho(T)=p/q$, with $p$ and $q$ coprime. Then every periodic orbit has period $q$ \cite[Proposition~11.1.5]{katok1995}. On each connected component $I$ of $\mathbb{S}^1\setminus\{T^q y = y\}$, the homeomorphism $T^q$ has no fixed point and moves every point in the
same direction. Thus, the forward $T^q$-orbit of each $x\in I$ converges to one endpoint of $I$, while the backward orbit converges to the other. See \cite[Proposition~11.2.2]{katok1995}. It follows that the forward orbit closure of $x$ supports exactly the uniform measure on its $\omega$-limit periodic orbit. This proves pointwise unique ergodicity and the formula
for $\Phi$.

Every invariant probability measure is supported on the periodic set. Thus $T$ is uniquely ergodic precisely when there is only one periodic orbit, in which case $\Phi$ is constant. If every point is periodic, then $T^q=\operatorname{id}_{\mathbb{S}^1}$ and
\[
\Phi(x)=\frac1q\sum_{j=0}^{q-1}\delta_{T^jx},
\]
which depends weak* continuously on $x$. In the remaining case, the periodic set contains at least two orbits and the complementary intervals contain nonperiodic points. The assignment of a point to its limiting periodic orbit jumps at a boundary between two periodic basins, so $\Phi$ is discontinuous.
\end{proof}

\begin{remark}
The choice of the acting semigroup is essential. Consider instead the bilateral action $\Z\acts\mathbb{S}^1$, $n\cdot x:=T^n x$. If $\rho(T)=p/q$ and a nonperiodic point lies in a component $I=(a,b)$ of $\mathbb{S}^1\setminus\{T^q y = y\}$ whose endpoints belong to distinct periodic orbits, then its bilateral orbit closure supports both $m_{\omega(a)}$ and $m_{\omega(b)}$. It is therefore not uniquely ergodic. Hence, the $\Z$-action generated by $T$ is pointwise uniquely ergodic if
and only if it is continuously pointwise ergodic, and this occurs precisely in the following cases:
\begin{enumerate}
\item $\rho(T)$ is irrational.
\item $\rho(T)=p/q$ and $T$ has exactly one periodic orbit.
\item $\rho(T)=p/q$ and $T^q=\operatorname{id}_{\mathbb{S}^1}$.
\end{enumerate}
\end{remark}

For $\alpha\in\R/\Z$, let $R_\alpha\colon\mathbb{S}^1\to\mathbb{S}^1$, $R_\alpha(x):=x+\alpha$, and denote by $\kappa_\alpha$ the Koopman representation of the $\Z$-action generated by $R_\alpha$. For rigid rotations, the forward and bilateral orbit closures coincide.

\subsubsection{Irrational rotations}
\label{ex:irrational-rotation}
Assume that $\alpha\notin\mathbb Q/\Z$. Then every orbit is dense and the unique invariant probability measure is Haar measure $m_{\mathbb{S}^1}$. Hence $\Phi(x)=m_{\mathbb{S}^1}$ for every $x\in\mathbb{S}^1$, and the ergodic quotient $\mathbb{S}^1/\Phi$ consists of a single point. The mean ergodic projection is $P_\alpha f =
\left(\int_{\mathbb{S}^1}f\,dm_{\mathbb{S}^1}\right)\mathbbm{1}$, and Theorem~\ref{thm:main1} gives
\[
\Fix(\kappa_\alpha)=\R\mathbbm{1} \qquad \text{and} \qquad \wCob(\kappa_\alpha)
=
\left\{
f\in C(\mathbb{S}^1):
\int_{\mathbb{S}^1}f\,dm_{\mathbb{S}^1}=0
\right\}.
\]

\subsubsection{Rational rotations}
\label{ex:rational-rotation}
Assume that $\alpha=\frac{p}{q}$, where $p$ and $q$ are coprime. Every orbit has exactly $q$ points and
\[
\Phi(x) = \frac{1}{q}\sum_{j=0}^{q-1}\delta_{x+j/q}.
\]
As $R_{p/q}^q = \operatorname{id}_{\mathbb{S}^1}$, Proposition~\ref{prop:circle-homeomorphisms} guarantees that the $\N$-action is continuously pointwise ergodic, and therefore the $\Z$-action as well. The map $m_q\colon\mathbb{S}^1\to\mathbb{S}^1$, $m_q(x):=qx$, is constant precisely on the fibers of $\Phi$ and induces a continuous homeomorphism $\overline{m}_q\colon
\mathbb{S}^1/\Phi \to \mathbb{S}^1$. The compact version of Theorem~\ref{thm:main-quotient} gives a homeomorphism $\widetilde\Phi\colon \mathbb{S}^1/\Phi \to E_1^+(\mathbb{S}^1,\Z)$. Therefore,
\[
E_1^+(\mathbb{S}^1,\Z)
\cong
\mathbb{S}^1/\Phi \cong \mathbb{S}^1.
\]

By Theorem~\ref{thm:main1}, the mean ergodic projection is
\[
P_\alpha f(x) = \int f \,d\Phi(x) = \frac{1}{q}\sum_{j=0}^{q-1}f\left(x+\frac jq\right),
\]
and
\[
\Fix(\kappa_\alpha) = \Fix(\kappa_{1/q})
=
\left\{
f\in C(\mathbb{S}^1):
f\left(x+\frac1q\right)=f(x)
\text{ for every }x\in\mathbb{S}^1
\right\}.
\]
Since the action is continuously pointwise ergodic, Theorem~\ref{thm:main1} gives $\wCob(\kappa_\alpha)=\ker P_\alpha$. In this case, this equality can be strengthened by removing the closure. Indeed, by  \cite[Lemma~3.1]{kocsard2013},
\[
\wCob(\kappa_\alpha)=\ker P_\alpha = \left\{
f\in C(\mathbb{S}^1):
\sum_{j=0}^{q-1}f\left(x+\frac jq\right)=0
\text{ for every }x\in\mathbb{S}^1
\right\} = \Cob(\kappa_\alpha).
\]

\begin{remark}
The same descriptions apply, through conjugacy, to every orientation-pre\-ser\-ving circle homeomorphism satisfying $T^q=\operatorname{id}_{\mathbb{S}^1}$.
\end{remark}

\subsection{Translations and escape to infinity}
\label{subsec:translations}

Let $S$ be a countable semigroup acting by proper maps on a separable locally compact metric space $X$. Assume that every orbit $Sx$ is closed and non-compact, and that each map $Sx\to Sx$, $y\mapsto sy$, is injective for every $s\in S$ and $x\in X$.

No orbit supports an invariant probability measure. Indeed, injectivity gives $\mu(\{sy\})=\mu(\{y\})$ for $s\in S$, $y\in Sx$. Since $Sy$ is infinite, every singleton must have zero mass, which contradicts countability of $Sx$ and $\mu(Sx)=1$. Thus every point escapes to infinity and $\Phi\equiv0$. Hence the action is continuously pointwise ergodic and vanishes at infinity. Its compactified ergodic map is constant, $\Phi_\infty\equiv\delta_\infty$, so the compactified ergodic quotient consists of a single point. Moreover, if $f\in\Fix(\kappa)$, then $f$ is constant on every orbit. Since each orbit is closed and non-compact, the condition $f\in C_0(X)$ forces this constant to be zero. Hence $\Fix(\kappa)=\{0\}$. By Theorem~\ref{thm:locally-compact-main}, $C_0(X)=\wCob(\kappa)$ and $M(X,S)=\{0\}$. If, in addition, $S$ is countable, discrete, bicancellative, and left
amenable, then Theorem~\ref{thm:main2} gives $A_{F_n}f\to 0$ uniformly on $X$ for every left F{\o}lner sequence $(F_n)_n$ and every $f\in C_0(X)$.

The basic examples are the translation actions $\Z\acts\Z$ and $\Z\acts\R$ defined by $n\cdot x:=x+n$, whose orbits are translates of $\Z$. More generally, an infinite countable bicancellative semigroup $S$ acts in this way on itself,
endowed with the discrete topology, by left translations. More generally, let $S$ be infinite, countable, discrete, and bicancellative, and let $S\acts Y$ be any action by proper maps. The diagonal action $s\cdot(y,t):=(sy,st)$ on $Y\times S$ has closed, discrete, non-compact orbits and therefore satisfies all the conclusions above. The projection $Y\times S\to Y$, $(y,t)\mapsto y$, is an equivariant factor map, but it is not proper. See Remark~\ref{remark:proper_factors}.

\subsection{A finite action without pointwise invariant measures}
\label{subsec:finite-example}

Let $X:=\{1,2,3,4\}$, and define $a,b\colon X\to X$ by $a(1)=a(2)=a(3)=1$, $a(4)=4$, and $b(1)=b(2)=2$, $b(3)=b(4)=4$. Let $S$ be the monoid generated by $a$ and $b$. Write $e_i:=\mathbbm{1}_{\{i\}}$. Since every element of $S$ fixes $4$, every coboundary vanishes at $4$, and hence $\Cob(\kappa) \subseteq \{f\in C(X):f(4)=0\}$. Conversely, $\kappa(b)e_1=0$, $\kappa(a)e_2=0$, $\kappa(a)e_3=0$, and therefore $e_1=e_1-\kappa(b)e_1$, $e_2=e_2-\kappa(a)e_2$, $e_3=e_3-\kappa(a)e_3$. Thus
\[
\Cob(\kappa) = \wCob(\kappa) = \{f\in C(X):f(4)=0\}.
\]
Moreover, $S3=X$, since $3$, $a3=1$, $b3=4$, and $b(a3)=2$ exhaust $X$. Hence every invariant function is constant, so $\Fix(\kappa)=\R\mathbbm{1}$. In consequence, as any $f \in C(X)$ can be expressed as $f = f(4)\mathbbm{1} + (f - f(4)\mathbbm{1})$,
\[
C(X) = \Fix(\kappa)\oplus\wCob(\kappa).
\]

On the other hand, $S1=S2=\{1,2\}$. If $\mu$ were an invariant probability measure on $\{1,2\}$, then $a_*\mu=\delta_1$ and $b_*\mu=\delta_2$, which is impossible. Thus the action is not pointwise invariant-measure admitting.

Finally, the full action has a unique invariant probability measure. Indeed, if $\mu\in M_1^+(X,S)$, then $\mu=a_*\mu$ implies $\operatorname{supp}(\mu)\subseteq\{1,4\}$, and $\mu=b_*\mu$ implies $\operatorname{supp}(\mu)\subseteq\{2,4\}$. Therefore, $\operatorname{supp}(\mu)\subseteq\{4\}$, and hence $M_1^+(X,S)=\{\delta_4\}$.

Thus the full action satisfies $C(X)=\Fix(\kappa)\oplus\wCob(\kappa)$ and has a unique invariant probability measure, but it is not pointwise invariant-measure admitting. This shows that the corresponding hypothesis in Theorem~\ref{thm:main1} cannot be omitted.

\subsection{A continuously pointwise ergodic system with discontinuous entropy}
\label{subsec:discontinuous-entropy}

Continuous pointwise ergodicity means that the unique invariant probability measure supported on the orbit closure of a point varies weak* continuously with the point. It is therefore natural to ask whether its measure-theoretic entropy must also vary continuously. The following example shows that this need not be the case, even for subshifts over a finite alphabet. Indeed, weak* continuity controls the integrals of continuous observables, while the entropy map for a finite alphabet subshift, although upper semicontinuous, need not be continuous \cite[Theorem~8.2]{Waltersbook}.

For a subshift $Y\subseteq A^{\Z}$, let $h_{\mathrm{top}}(Y)$ denote the topological entropy of the restricted shift $\sigma|_Y$, and, for $\nu\in M_1^+(Y,\Z)$, let $h_\nu(\sigma)$ denote the measure-theoretic entropy of $\sigma$ with respect to $\nu$.

Fix a strictly ergodic subshift $\Sigma\subseteq A^{\Z}$ with positive topological entropy. Such subshifts are known to exist (see, for example, \cite[Theorem~3]{bulatek1992}). Let $\mu$ be its unique invariant probability measure. Observe that, by the variational principle, $h_\mu(\sigma) > 0$.

We first construct periodic orbit subshifts converging to $\Sigma$ in the Hausdorff topology (see \cite{pavlov2023} for related constructions). For each $n \geq 1$, let $G_n$ denote the order-$n$ Rauzy graph associated with $\Sigma$. Its vertices are
the words in $\mathcal L_n(\Sigma)$, its edges are the words in $\mathcal L_{n+1}(\Sigma)$, and the incidence maps are given by the usual prefix and suffix maps (see \cite{rauzy1983}). Minimality of $\Sigma$ implies that this graph is strongly connected. Choose a closed directed walk traversing every edge, and let $x_n$ be the periodic point obtained by repeating the corresponding cyclic word. If $\Sigma_n:=\Z x_n$, then $\mathcal{L}_{n+1}(\Sigma_n)=\mathcal{L}_{n+1}(\Sigma)$. In consequence, $\Sigma_n\longrightarrow\Sigma$ in the Hausdorff topology. Since $\Sigma$ is not a periodic orbit, infinitely many distinct periodic orbits occur among the sets $\Sigma_n$. Passing to a subsequence and relabeling, we may therefore assume that the $\Sigma_n$ are pairwise disjoint. Moreover, none of the $\Sigma_n$ intersects $\Sigma$: otherwise $\Sigma$ would contain a periodic point, and minimality would force $\Sigma$ itself to be a finite periodic orbit, contradicting its positive topological entropy.

Set $X := \Sigma\cup\bigcup_{n\geq 1}\Sigma_n$. Since $\Sigma_n\to\Sigma$ in the Hausdorff topology, the set $X$ is compact. It is also shift-invariant, and hence is itself a subshift. Let $\mu_n$ be the unique invariant probability measure on the
periodic orbit $\Sigma_n$. We claim that $\mu_n \rightarrow\mu$ weak*. Indeed, every weak* cluster point $\nu$ of $(\mu_n)_n$ is
shift-invariant. Moreover, if $f\in C(A^{\Z})$ vanishes on $\Sigma$, by \cite[Maximum theorem, p.116]{berge1997}, then Hausdorff convergence gives
\[
\sup_{x\in\Sigma_n}|f(x)|\longrightarrow \sup_{x\in\Sigma}|f(x)| = 0,
\]
and hence $\int f\,d\nu=0$. Thus $\operatorname{supp}(\nu)\subseteq \Sigma$. Unique ergodicity of $\Sigma$ implies $\nu=\mu$.

The action $\Z\acts X$ is pointwise uniquely ergodic, with ergodic map
\[
\Phi(x)
=
\begin{cases}
\mu_n,&x\in\Sigma_n,\\
\mu,&x\in\Sigma.
\end{cases}
\]
The convergence $\mu_n\to\mu$, together with the fact that each $\Sigma_n$ is clopen in $X$, shows that $\Phi$ is weak* continuous. Thus the system is continuously pointwise ergodic. Let
\[
h\colon M_1^+(X,\Z)\longrightarrow[0,\infty),
\qquad
h(\nu):=h_\nu(\sigma),
\]
be the entropy map. Define the entropy function associated with the
ergodic map $\Phi$ by
\[
h_\Phi\colon X\longrightarrow[0,\infty),
\qquad
h_\Phi(x)=h_{\Phi(x)}(\sigma).
\]
Since $X$ is a subshift over a finite alphabet, the entropy map $h$ is upper semicontinuous in the weak* topology \cite[Theorem~8.2]{Waltersbook}. Since $\Phi$ is weak* continuous, it follows that $h_\Phi=h\circ\Phi$ is upper semicontinuous.

However, $h_\Phi$ is not continuous. Indeed, each $\mu_n$ is supported on a periodic orbit, and hence $h_{\mu_n}(\sigma)=0$. On the other hand, $h_\mu(\sigma) > 0$. For every $x\in\Sigma$, Hausdorff convergence allows us to choose $y_n\in\Sigma_n$ such that $y_n\to x$. For any such sequence,
\[
h_\Phi(y_n)=h_{\mu_n}(\sigma)=0,
\qquad\text{while}\qquad
h_\Phi(x)=h_\mu(\sigma)>0.
\]
Therefore, $h_\Phi$ is upper semicontinuous but not continuous.

\subsection{Topological emergence and the ergodic quotient}
\label{subsec:emergence}

Let $X$ be a compact metric space, let $S\acts X$ be continuously pointwise ergodic, and fix a metric $d_*$ on $M_1^+(X)$ compatible with the weak* topology. Since $E_1^+(X,S)=\Phi(X)$ and $\Phi$ is weak* continuous, the set $E_1^+(X,S)$ is compact.

Extending the definition of Berger and Bochi \cite[Definition~0.1]{berger2021} from a single continuous transformation to semigroup actions, define the \textbf{topological emergence} of $S\acts X$ at scale $\varepsilon>0$ by
\[
\mathcal E_{\mathrm{top}}(\varepsilon)
:=
N_{M_1^+(X)}
\left((E_1^+(X,S),d_*,\varepsilon\right)),
\]
where, for a subset $K$ of a metric space $(Z,d)$, $N_Z(K,d,\varepsilon)$ denotes the smallest number of $d$-balls of radius $\varepsilon$ in $Z$ whose union covers $K$.

\begin{corollary}
\label{cor:emergence-quotient}
Under the assumptions above,
\[
\dim_{\mathrm{top}}(X/\Phi)
\leq
\liminf_{\varepsilon\downarrow0}
\frac{\log\mathcal E_{\mathrm{top}}(\varepsilon)}
     {-\log\varepsilon},
\]
where $\dim_{\mathrm{top}}$ denotes the Lebesgue covering dimension.
\end{corollary}

\begin{proof}
By the compact version of Theorem~\ref{thm:main-quotient}, $X/\Phi\cong E_1^+(X,S)$. Therefore,
\[
\dim_{\mathrm{top}}(X/\Phi)
=
\dim_{\mathrm{top}}\left((E_1^+(X,S)\right)).
\]
For every compact metric space $K$, the standard inequalities
\[
\dim_{\mathrm{top}}K \leq \dim_{\mathrm H}K \leq \underline{\dim}_{\mathrm B}K
\]
hold; see
\cite[Theorem~2.11 and Lemma~3.3(v)]{robinson2011}. Moreover, relative and intrinsic covering numbers determine the same lower box-counting dimension \cite[Remark~1.2]{berger2021}. Thus,
\[
\underline{\dim}_{\mathrm B}
\left((E_1^+(X,S),d_*\right))
=
\liminf_{\varepsilon\downarrow0}
\frac{\log\mathcal E_{\mathrm{top}}(\varepsilon)}
     {-\log\varepsilon},
\]
and the conclusion follows.
\end{proof}

\bibliographystyle{abbrv}
\bibliography{references}

\end{document}